\documentclass[twoside]{article}
\usepackage[accepted]{aistats2020_arxiv}
\usepackage{xcolor}
\usepackage{chngcntr}
\usepackage{algorithm}
\usepackage{algpseudocode}
\usepackage{subcaption}
\usepackage{graphicx}
\usepackage{enumitem}
\usepackage{textcomp}

\usepackage[round]{natbib}
\usepackage{xr}
\makeatletter
\newcommand*{\addFileDependency}[1]{% argument=file name and extension
  \typeout{(#1)}
  \@addtofilelist{#1}
  \IfFileExists{#1}{}{\typeout{No file #1.}}
}
\makeatother
\newcommand*{\myexternaldocument}[1]{%
    \externaldocument{#1}%
    \addFileDependency{#1.tex}%
    \addFileDependency{#1.aux}%
}
\myexternaldocument{appendix}

\usepackage{tikz,calc}
\newcounter{Angle}

\DeclareMathOperator{\po}{po}
\usepackage{amsmath,amsfonts,amsthm,amssymb}
\usepackage{bbm}
\usepackage{bm}

\newtheorem{defn}{Definition}
\newtheorem{lemma}{Lemma}
\newtheorem{example}{Example}
\newtheorem{conjecture}{Conjecture}
\newtheorem{thm}{Theorem}
\newtheorem{prop}{Proposition}

\newcommand{\ci}{\mathrel{\perp\mspace{-10mu}\perp}}

\newcommand{\incomp}{\not\lessgtr}

\newcommand{\cH}{\mathcal{H}}
\newcommand{\cI}{\mathcal{I}}
\newcommand{\cM}{\mathcal{M}}
\newcommand{\cN}{\mathcal{N}}

\newcommand{\cP}{\mathcal{P}}

\newcommand{\bbP}{\mathbb{P}}

\DeclareMathOperator{\pa}{pa}
\DeclareMathOperator{\an}{an}

\DeclareMathOperator{\spo}{sp}

\DeclareMathOperator{\pre}{pre}
\DeclareMathOperator{\skel}{skel}

\newcommand{\leftarrowcirc}{\kern1.5pt\hbox{$\leftarrow$}\kern-1.5pt\hbox{$\circ$}\kern1.5pt}
\newcommand{\rightarrowcirc}{\kern1.5pt\hbox{$\circ$}\kern-1.5pt\hbox{$\rightarrow$}}
\newcommand{\circedge}{\kern1.5pt\hbox{$\circ$}\kern-1.5pt\hbox{$-$}\kern-1.5pt\hbox{$\circ$}\kern1.5pt}
\newcommand{\notadj}{\not\sim}

\begin{document}

\runningtitle{Ordering-Based Causal Structure Learning in the Presence of Latent Variables}


\twocolumn[

\aistatstitle{
Ordering-Based Causal Structure Learning \\in the Presence of Latent Variables
}

\aistatsauthor{ Daniel Irving Bernstein$^*$ \And Basil Saeed$^*$ \And Chandler Squires$^*$ \And Caroline Uhler}
\aistatsaddress{ MIT \And MIT \And MIT \And MIT}
]

\begin{abstract}
  We consider the task of learning a causal graph in the presence of latent confounders given i.i.d.~samples from the model. While current algorithms for causal structure discovery in the presence of latent confounders are constraint-based, we here propose a hybrid approach. We prove that under assumptions weaker than faithfulness, any sparsest independence map (IMAP) of the distribution belongs to the Markov equivalence class of the true model. This motivates the \emph{Sparsest Poset} formulation - that posets can be mapped to minimal IMAPs of the true model such that the sparsest of these IMAPs is Markov equivalent to the true model. Motivated by this result, we propose a greedy algorithm over the space of posets for causal structure discovery in the presence of latent confounders and compare its performance to the current state-of-the-art algorithms FCI and FCI+ on synthetic data. %, as well as an order of magnitude reduction in runtime for small graphs.
\end{abstract}

\section{INTRODUCTION}
Determining the causal structure between variables from observational data of these variables is a central task in many applications \citep{friedman2000using,robins2000marginal,heckerman1995real}. Causal structure is often modelled by a directed acyclic graph (DAG), where the nodes are associated with the variables of interest and the edges represent the direct causal effects these variables have on one another. In most realistic settings, only some of the variables in an environment are observed at any given time, i.e., only partial observations are available, leading to confounding effects on the observed variables. 
In such settings, a class of mixed graph models, called \emph{maximal ancestral graphs} (\emph{MAGs}) containing directed edges (representing direct causal effects), bidirected edges (representing the effect of a latent confounder on two variables) and undirected edges (representing selection bias), have been proposed to model the structure among the observed variables \citep{richardson2002ancestral}. 
In this paper, we concentrate on latent confounders and are concerned with the recovery of mixed graphs containing directed and bidirected edges.

Current methods for estimating MAGs are constraint-based generalizing the prominent PC algorithm for estimating DAGs in the fully observed setting \citep{spirtes2000causation}. This includes the \emph{Fast Causal Inference} (\emph{FCI}) algorithm \citep{spirtes2000causation} and its variants: the \emph{Really Fast Causal Inference} (\emph{RFCI}) algorithm \citep{colombo2012learning}, and the \emph{FCI+} algorithm \citep{claassen2013learning}. These methods depend on the faithfulness assumption to guarantee soundness and completeness, which has been shown to be restrictive \citep{uhler2013geometry}. 
In settings without latent confounders, studies have shown that score-based approaches, including the prominent GES algorithm \citep{chickering2002optimal}, achieve superior performance to constraint-based approaches~\citep{nandy2018high}. 
In purely constraint-based approaches such as PC, mistakes made in early stages of the algorithm tend to propagate and lead to later mistakes. Score-based approaches (which are usually greedy) are often more resilient to error propagation, since early mistakes only affect the local structure of the search space but do not affect the scores of later graphs. 
This motivates the development of an algorithm for causal structure discovery in the presence of latent confounders that shares this resilience with score-based approaches.

In this paper, we propose the \emph{sparsest poset (SPo) algorithm} for causal structure discovery in the presence of latent confounders. Since this algorithm uses both a scoring criterion and conditional independence testing to learn the model, we refer to it as a \emph{hybrid} method. The key idea that we use is that every MAG containing only directed and bidirected edges is consistent with a \emph{partial order} of the observed variables (\emph{poset}) and hence the  problem of causal structure discovery can be recast as the problem of learning a poset.
In particular, our main contributions are as follows:
\begin{itemize}%[noitemsep,nolistsep]
    \item We define a map that associates  to each partial order of the observed variables a MAG, so that the sample-generating distribution is Markov~to~it.
    %\vspace{-0.1cm}
    \item We prove that the sparsest such MAG is Markov equivalent to the true graph under conditions that are strictly weaker than faithfulness.
     %   \vspace{-0.1cm}
    \item We propose a greedy search over the space of posets based on the legitimate mark changes   by \cite{zhang2012transformational}  to move effectively between MAGs associated with different posets to find the poset yielding the sparsest graph.
      %  \vspace{-0.1cm}
    \item  By comparing the performance and speed of our algorithm to FCI and FCI(+) on synthetic data, we show that it is competitive to current stat-of-the-art methods for causal structure discovery with latent confounders.
\end{itemize}

\section{PRELIMINARIES AND RELATED WORK}
\label{section:background}
In the following, we review relevant concepts and related work; see also Appendix~\ref{appendix:graph-theory}.

 %\textcolor{red}{This needs to be rewritten quite a bit. Add references, don't use ``we say'' for other people's definitions. Connection between distribution and graph? Remove S-connectedness, but mention m-connectedness (no need to define, but give reference) under Markov equivalence subsection.}

\subsection{Directed Maximal Ancestral Graphs}
\label{section:maximal-ancestral}
All graphs in this paper can have directed and bidirected edges.
Let $G = (V,D,B)$ be a graph with vertices $V$, directed ($\rightarrow$) edges $D$, and bidirected ($\leftrightarrow$) edges $B$. 
We use $\skel(G)$ to denote the \emph{skeleton} of $G$, i.e., the undirected graph obtained by replacing all edges with undirected edges. We denote the number of edges of $G$ by $|G| := |D| + |B|$.
We use $\pa_G(i)$, $\spo_G(i)$, and $\an_G(i)$ respectively to denote the parents, spouses, and ancestors of a node $i$ in $G$, where we use the typical definitions as in~\cite{lauritzen1996graphical}.
$G$ is said to be \emph{ancestral} if it has no directed cycles, and whenever there is a bidirected edge $i\leftrightarrow j$ in $G$, there is no directed path from $i$ to $j$ \citep{richardson2002ancestral}. While ancestral graphs have been defined to also allow for undirected edges, we restrict our treatment to ancestral graphs with only directed and bidirected edges, which we will call \emph{directed ancestral graphs}.

\cite{richardson2002ancestral} generalized the standard notions of $d$-separation and $d$-connectedness for DAGs (see e.g.~\citep{lauritzen1996graphical}) to $m$-separation and $m$-connectedness for ancestral graphs. We write $A \ci_G B \mid C$ to indicate that $A$ and $B$ are $m$-separated given $C$ in $G$. We denote the set of all $m$-separation relations of a graph $G$ by $\cI(G)$.
Unlike for DAGs, in the case of ancestral graphs it is possible to have a pair of non-adjacent vertices $i$ and $j$ without an $m$-separation relation of the form $i\ci_G j \mid S$ for any $S\subseteq V\setminus\{i,j\}$ (see~\cite{richardson2002ancestral}). 
An ancestral graph is \emph{maximal} if every non-adjacent pair $i$ and $j$ satisfies $i \ci_G j \mid S$ for some $S \subseteq V\setminus\{i,j\}$.
\cite{richardson2002ancestral} showed that associated to every graph $G$ is a unique maximal supergraph, denoted $\overline{G}$, with the same set of $m$-separation statements.
They also give an efficient procedure for computing $\overline{G}$ from $G$.
We refer to a directed ancestral graph that is maximal as a \emph{directed maximal ancestral graph} (\emph{DMAG}).

%\textcolor{red}{What is this?: \cite{sadeghi2013stable} shows that all graphs are Markov equivalent to an ancestral graph.}

\subsection{Markov Properties of DMAGs}

Given a DMAG $G=(V,D,B)$, we associate to each vertex $i\in V$ a random variable $X_i$ such that the random vector  $X_V = (X_i : i\in V)$ has joint distribution $\bbP$. This distribution can be connected to the separation relations in $G$ via the \emph{Markov property} \citep{richardson1999markov}; namely, $\mathbb{P}$ is \emph{Markov} with respect to the DMAG $G$ if every $m-$separation relation in $G$ implies the corresponding conditional independence relation in $\bbP$, i.e.
$$A\ci_G B \mid C \Rightarrow X_A\ci_\bbP X_B \mid X_C $$
for all disjoint $A,B,C\subseteq V,$ where $\ci_\bbP$ denotes independence in $\bbP$.
Denoting by $\cI(\mathbb{P})$ the set of all CI relations in $\bbP$, the Markov property is then equivalent to $\cI(G) \subseteq \cI(\bbP)$. In this case, $G$ is called an \emph{independence map} (\emph{IMAP}) of $\bbP$; 
$G$ is called a \emph{minimal IMAP} of $\bbP$ if there is no edge of $G$ that can be deleted while keeping $G$ both maximal and an IMAP of $\bbP$.

Graphs $G$ and $H$ are said to be \emph{Markov equivalent} if $\cI(G) = \cI(H)$. The set of all graphs that are Markov equivalent to a given $G$ will be denoted $\cM(G)$. \cite{spirtes1996polynomial} provided a combinatorial characterization of graphs in the same Markov equivalence class (MEC). 
To do this, they used the notion of \emph{discriminating paths for a vertex $k$}: a path $\gamma = \langle i, \ldots, k, j \rangle$ between non-adjacent $i$ and $j$ is \emph{discriminating for $k$} if every node between $i$ and $k$ is both a collider and a parent of $j$, and there is at least one node between $i$ and $k$. 
\cite{spirtes1996polynomial} show that $G$ and $H$ are Markov equivalent if and only if they have the same skeleta, the same v-structures, and if for any path $\gamma$ that is discriminating for $k$ in both $G$ and $H$, $k$ is a collider on $\gamma$ in $G$ if and only if $k$ is a collider on $\gamma$ for $H$. 

% CHANDLER: JUST CHECKED BELOW AND SPIRTES SEEMS RIGHT
% \textcolor{red}{what is the right reference, Richardson or Spirtes? Contradicting statements.}

\cite{zhang2012transformational} provided a transformational characterization for the Markov equivalence class of a DMAG that will play an essential role in this paper. For this, they called the transformation of the edge $i \rightarrow j$ in $G$ into $i \leftrightarrow j$, or of the edge $i \leftrightarrow j$ to $i \rightarrow j$ a \emph{legitimate mark change} if there is no other directed path from $i$ to $j$ in $G$, $\pa_G(i) \subseteq \pa_G(j)$, $\spo_G(i) \setminus \{j\} \subseteq \pa_G(j) \cup \spo_G(j)$, and there is no discriminating path for $i$ on which $j$ is the endpoint adjacent to $i$. They showed that $G$ and $H$ are Markov equivalent if and only if there is a sequence of legitimate mark changes from $G$ to $H$.

% \textcolor{red}{combine the following two subsections and change the ordering. Motivate developing score-based approaches in the latent confounder setting. Then review GSP and say we will develop a latent confounder extension.}

\subsection{Causal Structure Discovery Algorithms}
\label{section:structure-learning}

%\textcolor{red}{To basil: add 1 sentence describing the problem of causal structure discovery in the setting of latent confounders}
The problem of causal structure discovery in the setting of latent confounders is to recover the Markov equivalence class of the underlying DMAG $G^*$ from samples on the observed variables. In particular, when the sample size $n\to\infty$, the problem is to recover the Markov equivalence class of the DMAG $G^*$ from $\cI(\mathbb{P})$. The most prominent existing algorithms for learning DMAGs\footnote{In fact, all of these methods are able to estimate \emph{MAGs}, which may include undirected edges to model selection bias.} are the Fast Casual Inference (FCI) algorithm \citep{spirtes2000causation} and its variants, most notably FCI+ \citep{claassen2013learning}, which has polynomial time complexity for sparse graphs while retaining large-sample consistency.
All of these methods are constraint-based; they start by estimating the skeleton of the graph based on the results of CI tests, then use the results of those CI tests to determine some edge orientations. However, constraint-based methods require the faithfulness assumption \citep{zhang2002strong}, which is restrictive in practice, and faithfulness violations lead to the removal of too many edges~\cite{uhler2013geometry}.

% The fundamentals of DAG models and Markov equivalence of DAGs is reviewed in Appendix \ref{appendix:dag-models}.

% \citep{chickering2002optimal} extended the transformational characterization from the case where $G \sim H$ to the case where $H \geq G$. Following \citep{solus2017consistency}, we say that $G_0, G_1, \ldots, G_M$ is a \emph{Chickering sequence} if $G_{i+1}$ is produced from $G_i$ by an edge addition or a covered edge reversal. \citep{chickering2002optimal} showed that for two DAGs $G$ and $H$ such that $G \leq H$, there exists a Chickering sequence with $G_0 = G$ and $G_M = H$.

In the DAG setting (i.e., no latent confounders) it has been shown that score-based approaches may require weaker assumptions for consistency \citep{van2013ell_, raskutti2018learning} and usually achieve superior performance for a given sample size~\citep{nandy2018high}.
This motivates the development of an algorithm for causal structure discovery that shares these properties with score-based approaches, and works in the presence of latent confounders.
Existing score-based approaches that can handle latent confounders require parametric assumptions.
For example, \cite{shpitser2012parameter} requires discreteness and
\cite{tsirlis2018scoring,nowzohour2017distributional} requires Gaussianity.

A particular approach that will play an important role in this paper is the \emph{Sparsest Permutation} algorithm, 
introduced in \cite{raskutti2018learning},
which associates to each permutation $\pi$ a DAG $G_\pi$, which is a minimal IMAP of the data-generating distribution.
The Sparsest Permutation algorithm is a hybrid method, combining aspects of the constraint- and score-based paradigms. 
Like many constraint-based methods, it does not require parametric assumptions,
and like many score-based methods, it seems resilient to error-propagation.
Since under restricted faithfulness assumptions the sparsest such $G_\pi$ is Markov equivalent to the true DAG $G^*$, this motivates a greedy search over the space of permutations to determine the sparsest $G_\pi$. In fact, in~\cite{solus2017consistency} the authors proved that starting in any minimal IMAP there exists a sequence of minimal IMAPs connecting it to the true DAG $G^*$ by legitimate mark changes such that the number of edges is weakly decreasing. Hence the \emph{Greedy Sparsest Permutation} (\emph{GSP}) algorithm is consistent for causal structure discovery in the fully observed setting. 

In the following section, we generalize the sparsest permutation algorithm to the setting with latent confounders by using posets instead of permutations. In particular, we show that under restricted faithfulness assumptions the DMAG associated with the \emph{Sparsest Poset} is Markov equivalent to the true DMAG. This motivates the introduction of a greedy search over posets, which we term \emph{Greedy Sparsest Poset} (\emph{GSPo}) algorithm and introduce in Section \ref{section:greedy}. Finally, in Section \ref{section:simulation} we analyze its performance and compare it to the FCI algorithms on synthetic data.

%Moreover, \cite{solus2017consistency} showed that from any minimal IMAP $H$ of $G$, there exists a sequence of minimal IMAPs from $H$ to $G$ such that the number of edges is weakly decreasing, suggesting a greedy algorithm over the space of minimal IMAPs known as the \emph{Greedy Sparsest Permutation} algorithm. Leveraging the similarity between graphs at each step to limit the number of conditional independence tests and the size of conditioning sets, the authors show that the algorithm enjoys high-dimensional consistency guarantees for sparse graphs. In Section \ref{section:greedy}, we introduce a greedy version of the SPo algorithm, known as the \emph{Greedy Sparsest Poset} (\emph{GSPo}) algorithm, and in Section \ref{section:simulation} we demonstrate its performance on synthetic data.

\section{SPARSEST POSET}
\label{section:sparsity}
% \textcolor{red}{I would change the order in this section and only have 2 subsections, one on faithfulness and one on sparsest posets (which starts with posets and builds in the sparsest IMAP result just before the sparsest poset theorem.)} 
This section contains our main results. We first introduce the restricted faithfulness notion required for our results and show that it is strictly weaker than the standard faithfulness assumption. 
%we describe our results on the sparsest IMAPs of a distribution $\bbP$ that is restricted-faithful to a DMAG $G^*$.
%In \ref{subsec:faithfulness}, we introduce the notion of faithfulness required for our results and show that it is strictly weaker than the faithfulness assumption. 
Then we introduce a map from posets to DMAGs which are  minimal IMAPs of the data-generating distribution, and show that the sparsest DMAG in the image of this map is Markov equivalent to the true DMAG $G^*$.

\subsection{Restricted Faithfulness}
\label{subsec:faithfulness}

An important assumption for constraint-based methods to recover $G^*$ from $\cI(\bbP)$ is the \emph{faithfulness assumption}, which asserts that $\cI(\bbP) = \cI(G^*)$. In practice, this assumption is very sensitive to hypothesis testing errors for inferring CI relations from data and almost-violations are frequent~\citep{uhler2013geometry}. This motivates studying restricted versions of the faithfulness assumption~\cite{ramsey2012adjacency,raskutti2018learning}. In the following, we introduce a restricted faithfulness assumption for DMAGs, which we show is sufficient for learning DMAGs.

% \textcolor{red}{Since this subsection contains new material, it should appear here. Explain the importance of faithfulness assumptions for causal structure discovery. Then say that we are introducing a new restricted faithfulness assumption and we will show consistency under this assumption.} 

\begin{defn}
    A distribution $\bbP$ is \emph{restricted-faithful} to a DMAG $G = (V,D,B)$ if it is Markov to $G$ satisfying
    \begin{enumerate}
    \vspace{-0.2cm}
        \item \emph{Adjacency-faithfulness}: If $(i, j)\in B\cup D$, then $X_i \not\ci_\bbP X_j \mid X_S$ for any $S \subseteq V \setminus \{i,j\}$;
        \vspace{-0.1cm}
        \item \emph{Orientation-faithfulness}: If $i - k - j$ is contained in the skeleton of $G$ and $i$ is m-connected to $j$ given some subset $S \subseteq V \setminus \{i,j\}$, then $X_i \not\ci_\bbP X_j \mid X_S$.
        \vspace{-0.5cm}
        \item \emph{Discriminating-faithfulness}: If $\langle i, \dots, k, j\rangle$ is a discriminating path in $G$ and $i$ is m-connected to $j$ given some subset $S \subseteq V \setminus \{i,j\}$, then $X_i \not\ci_\bbP X_j \mid X_S$.
        \vspace{-0.2cm}
    \end{enumerate}
\end{defn}

\begin{figure}[!t]
    \vspace{-0.2cm}
    \begin{subfigure}{0.59\linewidth}
    % \begin{align*}
    %     X_1 &= \varepsilon_1\\
    %     X_2 &= X_1 + \varepsilon_2\\
    %     X_3 &= X_2 + \varepsilon_3\\
    %     X_4 &= X_1 + \varepsilon_4\\
    %     X_5 &= X_4 +\varepsilon_5\\
    %     X_6 &= X_3-X_4 + \varepsilon_6
    % \end{align*}
    \begin{align*}
        X_i &= X_{\pa_{G^*}(i)} + \varepsilon_i,~i \neq 6\\
        X_6 &= X_3-X_5 + \varepsilon_6
    \end{align*}
    \caption{Structural equation model}
    \label{fig:restrictedFaithfulModel}
    \end{subfigure}
    \begin{subfigure}{0.39\linewidth}
	%\begin{subfigure}[t]{.48\textwidth}
    \centering
    \includegraphics[width=\linewidth]{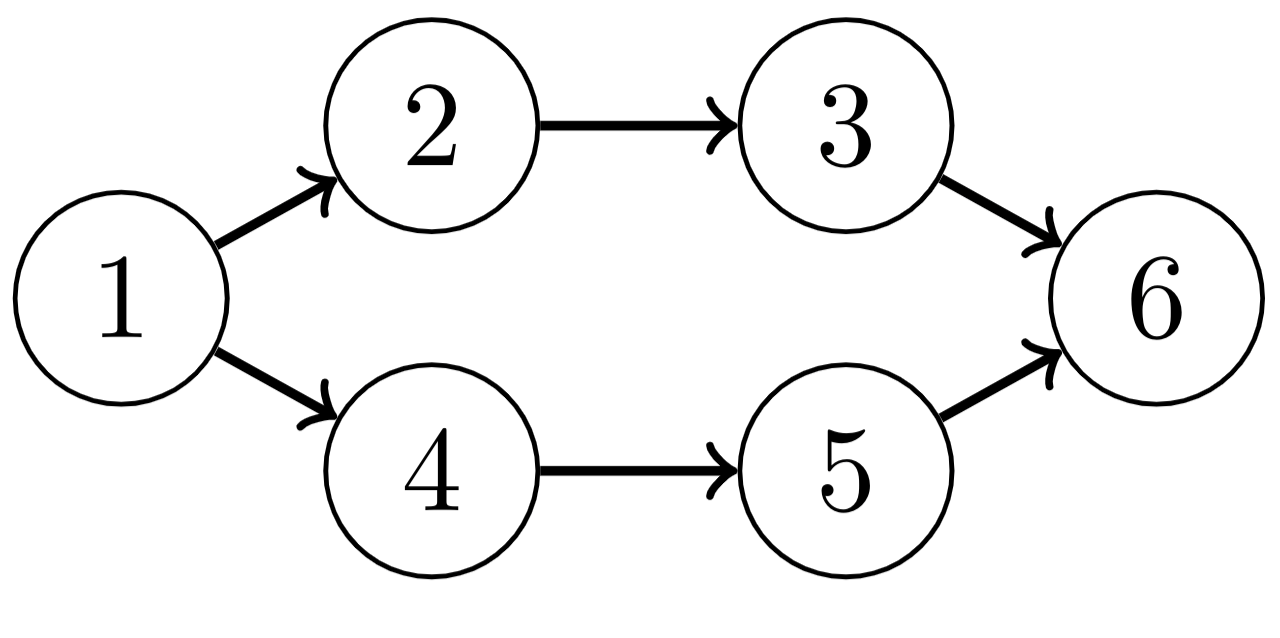}
    %\begin{tikzpicture}
    %    \node (a) at (0,-1.625){};
    %    \node (b) at (0,2){};
	   % \node[shape=circle,draw=black] (3) at (.6,.5) {3};
	   % \node[shape=circle,draw=black] (2) at (-.6,.5) {2};
	   % \node[shape=circle,draw=black] (1) at (-1.5,0) {1};
	   % \node[shape=circle,draw=black] (4) at (-.6,-.5) {4};
	   % \node[shape=circle,draw=black] (5) at (.6,-.5){5};
	   % \node[shape=circle,draw=black] (6) at (1.5,0){6};
%		\p%ath[->,style=thick] (1) edge (2);
%		\p%ath[->,style=thick] (2) edge (3);
%		\p%ath[->,style=thick] (3) edge (6);
%		\p%ath[->,style=thick] (1) edge (4);
%		\p%ath[->,style=thick] (4) edge (5);
		%\p%ath[->,style=thick] (5) edge (6);
    %\end{tikzpicture}
    \caption{$G^*$}
    \label{fig:restrictedFaithfulGraph}
    \end{subfigure}
    \caption{A structural equation model giving rise to a joint distribution $\bbP$ that is
    restricted-faithful, but not faithful, to the graph $G^*$.}\label{fig:restrictedFaithful}
    %\vspace*{-\baselineskip}
\end{figure}

% Given a distribution $\bbP$, $\bbP$ is \emph{faithful} to $H$ if $\cI(\bbP) = \cI(H)$.
% We say $\bbP$ is \emph{adjacency-faithful} to $H$ if $H$ is an IMAP of $\bbP$ and
% $$i \in \adj_H(j) \Rightarrow i \not\ci_\bbP j \mid S$$
% for any $S \subseteq V(G) \setminus \{i,j\}$.
% We say that $\bbP$ is \emph{orientation-faithful} to $H$ if $H$ is an IMAP of $\bbP$ and
% $i - k - j \in \skel(H)$ and $i$ is m-connected to $j$ given $S \subseteq V(G) \setminus \{i, j\}$, then $i \not\ci_\bbP j \mid S$.
% We say that $\bbP$ is \emph{discriminating-faithful} to $H$ if
% $H$ is an IMAP to $\bbP$ and whenever $\langle i, \dots, k, j\rangle$ is a discriminating path in $H$
% and $i$ is $S$-connected to $j$, then $i \not\ci_\bbP j | S$.
% We say that $\bbP$ is \emph{restricted-faithful} to $H$ if it is adjacency-faithful, orientation-faithful, and discriminating-faithful to $H$.

%\cite{raskutti2018learning} showed that this is sufficient to imply consistency of the sparsest permutation algorithm.
It is clear that faithfulness implies restricted-faithfulness.
Moreover, restricted-faithfulness is a strictly weaker condition -
there exist joint distributions $\bbP$ that are restricted-faithful to a DMAG
that are not faithful.
For example, let $\bbP$ be given by the structural equation model in Figure \ref{fig:restrictedFaithfulModel},
where each $\varepsilon_i \sim \cN(0, 1)$.
Then $\bbP$ is restricted-faithful, but not faithful to the graph $G^*$ displayed in Figure \ref{fig:restrictedFaithfulGraph}.
To see that $\bbP$ is not faithful to $G^*$, note that $X_1 \ci_\bbP X_6$ even though $1$ and $6$ are not m-separated in $G^*$.

\subsection{Sparsest Poset}
\label{subsec:sparsest-poset}

In this section, we show that the Markov equivalence class of a DMAG $G^*=(V,D^*,B^*)$  can be determined from $\cI(\bbP)$ under the restricted faithfulness assumption by  casting this problem into an minimization problem over the space of partial orders of the set $V$.
We do this by mapping the space of these partial orders to minimal IMAPs of $G^*$ and minimizing a cost that is a function of such an IMAP.

A \emph{partial order} on a set $V$ is a relation $\le$ on $V$
that is reflexive, transitive, and antisymmetric.
Two elements $i,j \in V$ are said to be \emph{incomparable}
if neither $i \le j$ nor $j \le i$ holds.
We denote this symbolically by $i \incomp j$.
A set $V$ equipped with a specified partial order $\le$ is called a \emph{partially ordered set} (\emph{poset}), denoted $(V,\le)$.
Then, $V$ is called the \emph{ground set} of the poset.
The \emph{empty poset} is the poset $(V,\le)$ such that all $i,j \in V$ are incomparable.
We denote the set of all posets with a ground set $V$ by $\mathcal{P}(V)$.
Given a poset $\pi = (V,\le)$ and $s_1,\dots,s_k \in V$, define
\[
    \pre_\pi(s_1,\dots,s_k) := \{x \in V: x \le s_i \textnormal{ for some } 1 \le i \le k\}.
\]

Associated to each directed ancestral graph $G=(V,D,B)$
is a partial order $\le_G$ on $V$, defined by 
$$i \le_G j \Leftrightarrow i\in\an_{G}(j).$$
Note that the ancestral property implies that if $i\leftrightarrow_G j$,
then $i \incomp_G j$.
We denote the poset $(V,\le_G)$ by $\po(G)$. The map $G\mapsto (V,\le_G)$ gives a bijection from the set of \emph{complete} DMAGs, i.e., DMAGs whose skeleta are complete graphs,
to $\cP(V)$, the set of posets with ground set $V$.
Since not all DMAGs are complete, the
set of DMAGs on $V$ is strictly larger than $\cP(V)$.

This relationship between ancestral graphs and posets motivates describing the sparsest IMAP of a distribution $\bbP$ that is restricted-faithful to a DMAG $G^*$ in terms of posets by mapping every poset to an IMAP. This will lead to the concept of \emph{sparsest posets}; the posets of $\cP(V)$ that are mapped to DMAGs in $\cM(G^*)$. To obtain the map, we need the following definition.

\begin{defn}\label{defn:firstConstruction}
Given a joint distribution $\bbP$ on the random vector $X_V$ and
a poset $\pi = (V,\le_\pi)$.
Define $AG(\pi,\bbP)$ as the ancestral graph with directed edge set
$$\{i\rightarrow j: i\le_\pi j, X_i\not\ci_\bbP X_j \mid X_{\pre_\pi({i,j})\setminus \{i,j\}} \}$$
and bidirected edge set
$$\{i\leftrightarrow j: i\incomp_\pi j, X_i\not\ci_\bbP X_j \mid X_{\pre_\pi({i,j})\setminus \{i,j\}}\}.$$
\end{defn}

When $\le_\pi$ is a \emph{total order}, i.e. a partial order where the relations $i \le_\pi j$ or $j \le_\pi i$ hold for all $i,j$,
then $AG(\pi,\bbP)$ defines a map from permutations to DAGs and is the one used in the GSP algorithm~\citep{raskutti2018learning}. The authors showed in this case that $AG(\pi,\bbP)$ is a minimal IMAP for $\bbP$
for all total orders $\le_\pi$.
Unfortunately, as shown in the following example, $AG(\pi,\bbP)$ may not be an IMAP of $\bbP$ when $\le_\pi$ is allowed to be an arbitrary partial order.

\begin{example}\label{ex:notAnIMAP}
    Let $\bbP$ be a joint distribution that is restricted-faithful to the DMAG $G^*$ shown in Figure~\ref{fig:trueGraph}.
    Let $\pi$ be the poset with ground set $\{1,2,3,4\}$ and relations $2 \le 3$, $1 \le 4$, and $i \incomp j$ otherwise.
    Then $AG(\pi,\bbP)$, shown in Figure \ref{fig:notAnIMAP}, is not an IMAP of $\bbP$.
    To see this, note that $4\ci_{AG(\pi,\bbP)}3 \mid \{2\}$,
    but $X_4 \not\ci_\bbP X_3 \mid \{X_2\}$ since
    $4\leftrightarrow 2 \leftarrow 1 \rightarrow 3$ is a $\{2\}$-connecting path in $G^*$.
    
    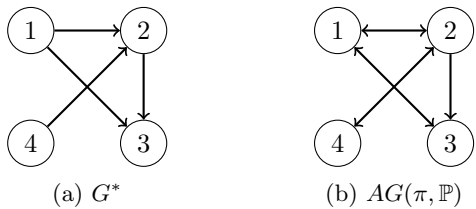
\begin{figure}[!t]
    \centering
	\begin{subfigure}[b]{0.48\linewidth}
	\centering
	\begin{tikzpicture}
		\node[shape=circle,draw=black] (1) at (-1.5, 0) {1};
		\node[shape=circle,draw=black] (4) at (-1.5, -1.5) {4};
		\node[shape=circle,draw=black] (2) at (0, 0) {2};
		\node[shape=circle,draw=black] (3) at (0, -1.5) {3};
		\path[->,style=thick] (1) edge (2);
		\path[->,style=thick] (1) edge (3);
		\path[->,style=thick] (2) edge (3);
		\path[->,style=thick] (4) edge (2);
	\end{tikzpicture}
	\caption{$G^*$}
	\label{fig:trueGraph}
	\end{subfigure}
	\begin{subfigure}[b]{0.48\linewidth}
	\centering
	\begin{tikzpicture}
		\node[shape=circle,draw=black] (1) at (-1.5, 0) {1};
		\node[shape=circle,draw=black] (4) at (-1.5, -1.5) {4};
		\node[shape=circle,draw=black] (2) at (0, 0) {2};
		\node[shape=circle,draw=black] (3) at (0, -1.5) {3};
			
		\path[<->,style=thick] (1) edge (2);
		\path[<->,style=thick] (1) edge (3);
		\path[->,style=thick] (2) edge (3);
		\path[<->,style=thick] (4) edge (2);
	\end{tikzpicture}
	\caption{$AG(\pi,\bbP)$}
	\label{fig:notAnIMAP}
	\end{subfigure}
% 	\caption{Graphs for Example~\ref{ex:notAnIMAP}. (a) shows the graph $G^*$ that $\bbP$ is faithful to. (b) shows $AG(\pi,\bbP)$, which in this case, is not an IMAP of $\bbP$.}
	\caption{Graphs for Example~\ref{ex:notAnIMAP}. $\bbP$ is faithful to $G^*$ but $AG(\pi, \bbP)$ is not an IMAP of $\bbP$.}
	%\vspace*{-\baselineskip}
	\end{figure}
\end{example}

However, we show in the following proposition, which is proven in Appendix~\ref{proof_prop1}, that one can construct a minimal IMAP of $\bbP$ for any poset $\pi$ using the map $AG(\cdot,\cdot)$ by defining
\[
    G_\pi^\bbP := \overline{AG(\po(AG(\pi,\bbP)),\bbP)}.
\]
where $\bbP$ and $\pi$ are as in Definition \ref{defn:firstConstruction}.
Recall that $\overline{G}$ denotes the maximal closure of $G$.
We want $G_\pi$ to be maximal since the results of \cite{zhang2012transformational}
regarding legitimate mark changes apply only to maximal ancestral graphs.
To simplify notation, we use $G_\pi$
instead of $G_\pi^\bbP$ when $\bbP$ is clear from context.

\begin{prop}
\label{prop:construction}
Let $\bbP$ be a joint distribution on $V$ that is restricted-faithful to a DMAG.
Then $G_\pi$ is a minimal IMAP of $\bbP$ for any poset $\pi\in\cP(V)$.
\end{prop}
As we show in the following example, including the maximal closure in the definition of $G_\pi$ is required since it may otherwise not be maximal.%, since $AG(\pi, \bbP)$ is not necessarily maximal.

\begin{example}\label{ex:notNecessarilyMaximal}
    Let $\bbP$ be a joint distribution faithful to the graph $G^*$ displayed in Figure \ref{fig:trueGraphNotMaximalExample}.
    Let $\pi$ be the poset with ground set $V=\{1,2,3,4,5\}$ and ordering relations $1 \le 2$, 
    $3 \le 4$, $5 \le 4$, and $i \incomp j$ otherwise.
    Then $AG(\po(AG(\pi,\bbP)),\bbP)$, displayed in Figure \ref{fig:badIMAPNotMaximalExample}, is not maximal.
    To see this, note that $AG(\po(AG(\pi,\bbP)),\bbP)$ lacks an edge between $2$ and $4$,
    while there is no set $S\subseteq V\setminus\{2,4\}$ that $m$-separates $2$ and $4$ in $AG(\pi(AG(\pi,\bbP),\bbP).$
    
    \begin{figure}
	\begin{subfigure}[b]{0.48\linewidth}
    \centering
	\begin{tikzpicture}
	    \foreach \i in {1, 2, 3, 4, 5} {
		    \setcounter{Angle}{\i * 360 / 5}
		    \node[shape=circle,draw=black] (\i) at (\theAngle:1) {\i};
	    }
		\path[->,style=thick] (4) edge (5);
		\path[->,style=thick] (1) edge (2);
		\path[->,style=thick] (3) edge (2);
		\path[->,style=thick] (3) edge (1);
		\path[->,style=thick] (3) edge (4);
		\path[<->,style=thick] (4) edge (1);
	\end{tikzpicture}
	\caption{$G^*$}
	\label{fig:trueGraphNotMaximalExample}
	\end{subfigure}
	\begin{subfigure}[b]{0.48\linewidth}
    \centering
	\begin{tikzpicture}
	    \foreach \i in {1, 2, 3, 4, 5} {
		    \setcounter{Angle}{\i * 360 / 5}
		    \node[shape=circle,draw=black] (\i) at (\theAngle:1) {\i};
	    }
		\path[->,style=thick] (5) edge (4);
		\path[->,style=thick] (1) edge (2);
		\path[->,style=thick] (3) edge (4);
		\path[<->,style=thick] (1) edge (5);
		\path[<->,style=thick] (2) edge (5);
		\path[<->,style=thick] (3) edge (5);
		\path[<->,style=thick] (1) edge (3);
		\path[<->,style=thick] (2) edge (3);
		\path[<->,style=thick] (1) edge (4);
	\end{tikzpicture}
	\caption{$AG(\po(AG(\pi,\bbP),\bbP)$}
	\label{fig:badIMAPNotMaximalExample}
	\end{subfigure}
% 	\caption{Graphs for Example~\ref{ex:notNecessarilyMaximal}. Subfigure (a) shows the graph $G^*$ that the distribution $\bbP$ is faithful to, while (b) shows $AG(\po(AG(\pi,\bbP),\bbP)$, which is not maximal in this case.}
	\caption{Graphs for Example~\ref{ex:notNecessarilyMaximal}. $\bbP$ is faithful to $G^*$ but $AG(\po(AG(\pi,\bbP),\bbP)$ is not maximal.}
	\end{figure}
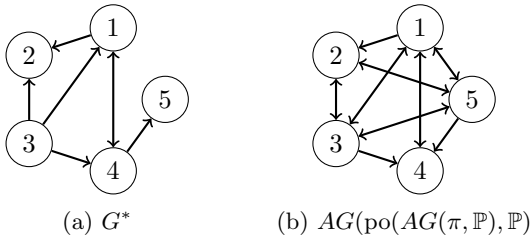
\end{example}

% Since $\{G_\pi : \pi\}$ is a strict subset of all MAGs over $V$, we must check that there is always some $\pi$ with $G_\pi = G^*$. The following proposition guarantees this for the set $P(G^*)$ of posets \emph{compatible} with $G^*$, where $\pi \in P(G^*)$ if 
% $$i\in\an_G(j) \Rightarrow i\prec_\pi j$$
% and
% $$i\leftrightarrow_G j \Rightarrow i\incomp_\pi j.$$
% \textcolor{red}{BASIL: If we do not want to use the concept of "compatible poset" anywhere else, then this proposition can be phrased in terms of $\po(G^*)$. The only other places $P$ might come up is if we want to say that all "compatible" posets map to the same $G$.} 

Having defined a map from posets to minimal IMAPs for DMAGs, we are almost ready to state our result on the consistency of the sparsest poset. The following theorem establishes that under restricted-faithfulness \emph{all} sparsest IMAPs of $G^*$ are Markov equivalent to $G^*$.

%\textcolor{red}{Note: I added $\skel$ and $|G|$ to 2.1.}

\begin{thm}
\label{thm:sparsest_imap}
Given a distribution $\bbP$ and a DMAG $G^*$ that is an IMAP of $\bbP$,
let 
\begin{equation}
    \label{eq:sparsest_imap}
    G \in \arg\min_{\{H :\;  H \textrm{ is an IMAP of }\bbP\}} |H|.
\end{equation}
\begin{enumerate}
    \item[(a)] If $\bbP$ is adjacency-faithful to $G^*$, then $\skel(G) = \skel(G^*)$.
    \item[(b)] If $\bbP$ is restricted-faithful to $G^*$, then $G \in \cM(G^*)$.
    
\end{enumerate}
\end{thm}

%That is, if $\bbP$ is restricted-faithful to $G^*$, then \emph{any sparsest IMAP}, i.e., any IMAP with the smallest number of edges, is Markov equivalent to $G^*$.

The proof of this theorem is given in Appendix~\ref{proof_thm1}; it involves using the adjacency faithfulness condition to obtain $\skel(G)\supseteq \skel(G^*)$ for any IMAP $G$. Then we show that the IMAP condition on $G$, under restricted-faithfulness of $\bbP$, forces a graph with the same skeleton as $G^*$ to have matching unshielded colliders and matching discriminating paths when these discriminating paths are present in both of these graphs.

% Not only does $G_\pi$ always produce an IMAP of $\bbP$,
%  if $\bbP$ is restricted-faithful to a given DMAG,
% then it is possible to recover that DMAG via our construction.
% This is the content of the following proposition, whose proof is in the appendix.

The following proposition establishes that $G^*$ is in the image of $\pi \mapsto G_\pi$; its proof is given in Appendix~\ref{proof_prop2}. Thus, when restricting our search over IMAPs to the the image of this map, the optimum is still in our feasible set.

\begin{prop}
\label{prop:gpi-eq-gstar}
    Let $\bbP$ be restricted-faithful to DMAG $G^*$. If $G\in\mathcal M(G^*)$,
    and $\pi = \po(G)$,
    then $G_\pi = G$.
\end{prop}

% In words, this asserts that for any $\bbP$ that is restricted faithful to a DMAG $G^*$, we will always have a poset that is mapped to $G^*$ via $G_\pi$. 

We are now ready to state our main result. 

\begin{thm}[Sparsest Poset]\label{thm:sparsestPoset}
Let $\bbP$ be a distribution on $V$ that is restricted faithful to  a DMAG $G^*$. 
%Let $\mathcal{P}(V)$ be the set of posets with ground set $V$. 
If
$$\tau \in \arg\min_{\pi\in\mathcal{P}(V)}|G_\pi|,$$
then $G_\tau$ is Markov equivalent to $G^*$.
\end{thm}

\begin{proof}
Propositions~\ref{prop:construction} and~\ref{prop:gpi-eq-gstar} together imply that there is an IMAP $H = G_\pi$ for some $\pi$ such that $|H| = |G^*|$.
Theorem~\ref{thm:sparsest_imap} then gives the desired result.
\end{proof}

\section{GREEDY SPARSEST POSET}
\label{section:greedy}

Theorem \ref{thm:sparsestPoset} formulates the problem of finding a graph $G^*$ from $\bbP$ as a discrete optimization problem over $\cP(V)$, the set of all posets on the ground set $V$. 
In this section, we discuss solving this optimization problem by imposing a graph structure on $\mathcal{P}(V)$
and then performing a greedy search along the edges of the graph.
Note that Theorem~\ref{thm:sparsestPoset} does not guarantee that a greedy approach returns an optimum. Supported by simulations, we will conjecture that this is indeed the case.

\subsection{Greedy Sparsest Poset}

\begin{figure*}[t!]
	\centering
	\begin{subfigure}[t]{.48\textwidth}
	    \includegraphics[width=\linewidth]{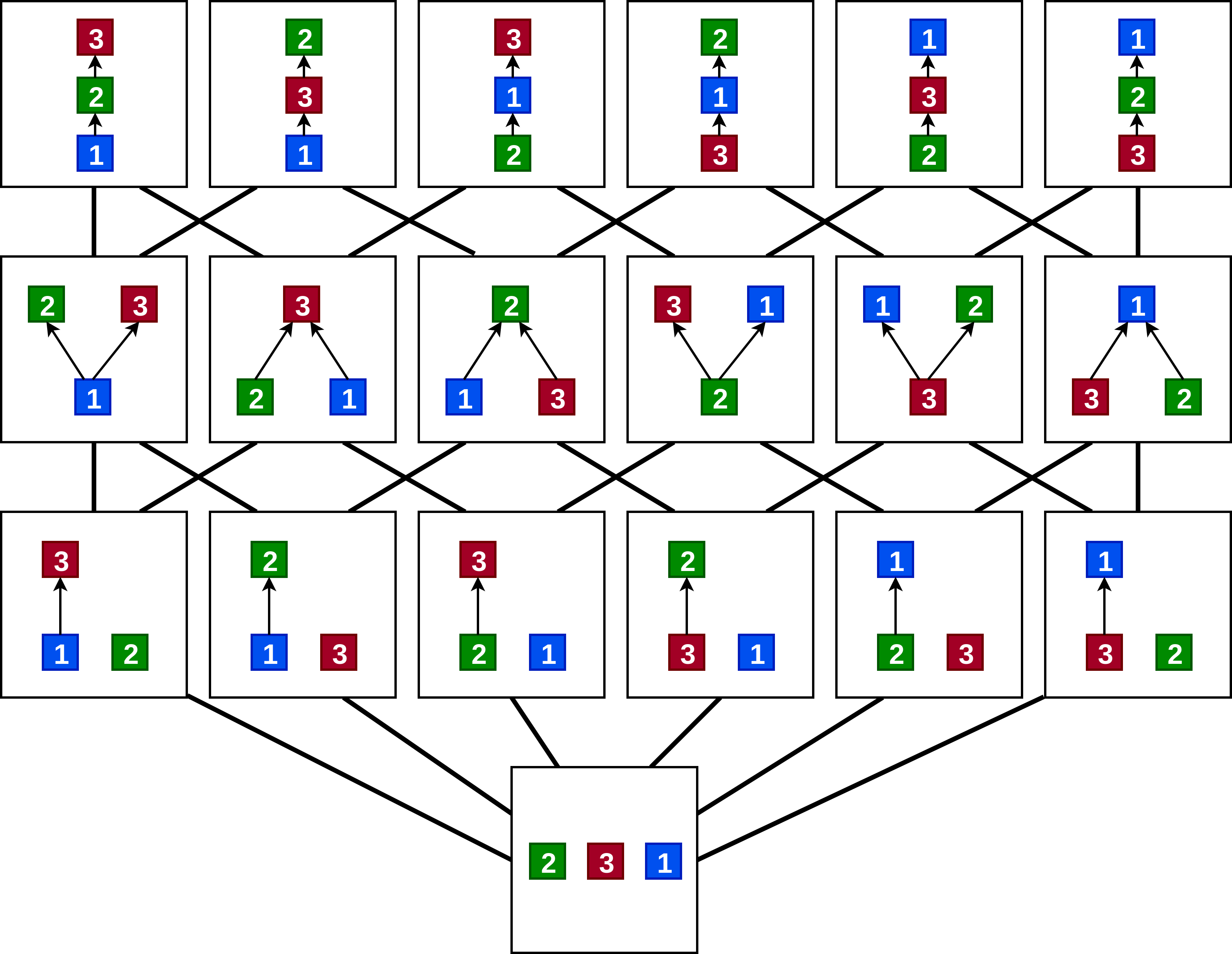}
        \caption{$\mathcal{H}_{\mathcal{P}(V)}$ for $V = \{1,2,3\}$.}
	    \label{fig:posetofposets}
	\end{subfigure}
	~
	\begin{subfigure}[t]{.48\textwidth}
	    \includegraphics[width=\linewidth]{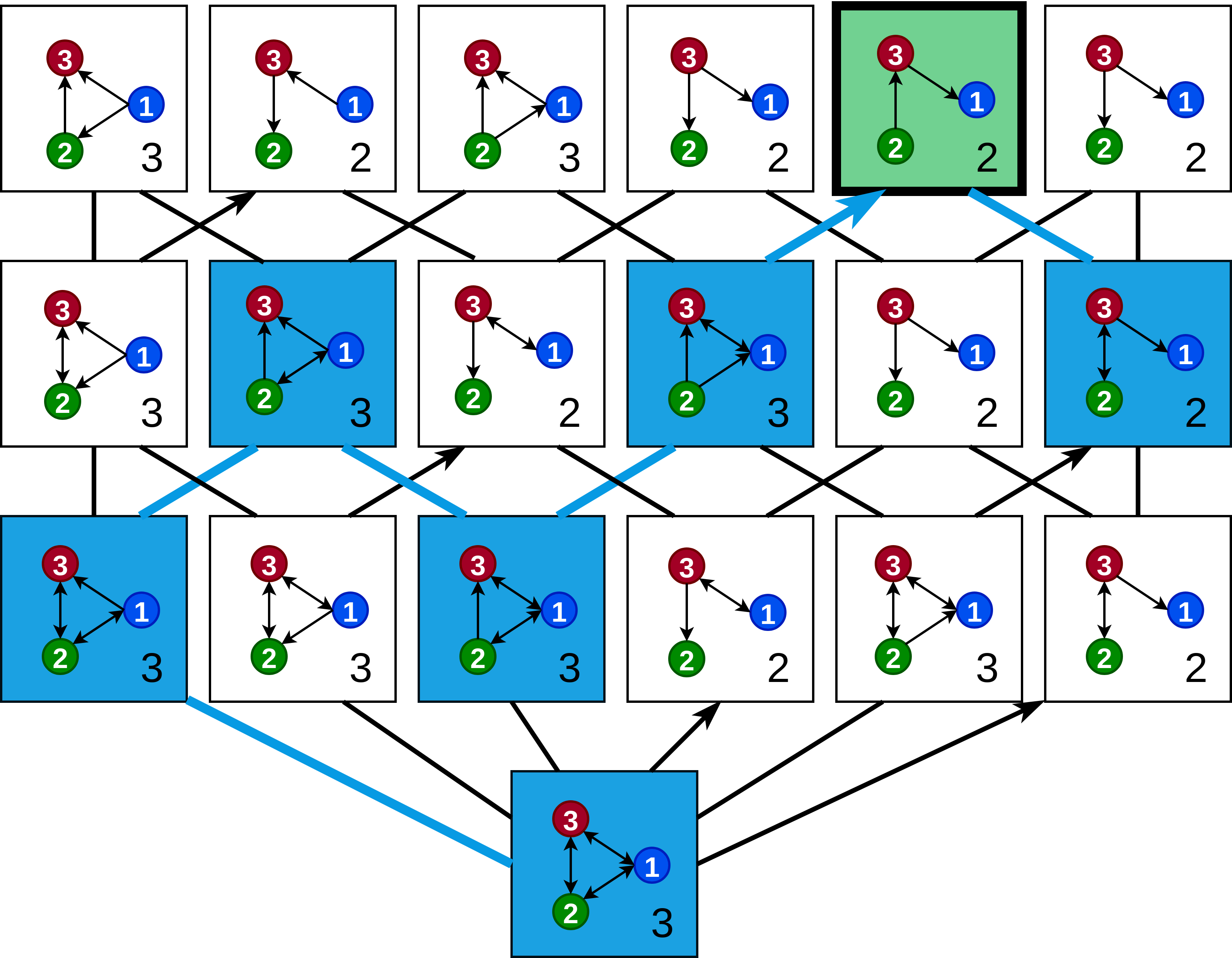}
	    \caption{A relabeling of $\mathcal{H}_{\mathcal{P}(V)}$ be replacing each $\pi$ with its corresponding $G_\pi$ when $\cI(\bbP) = \{X_1 \ci_\bbP X_2 \mid X_3\}$.}
	    \label{fig:posetofmags}
	\end{subfigure}
	\caption{
	(a) shows $\cH_{\cP(V)}$ for $V=\{1,2,3\}$.
	Each of the large squares represents a poset $\pi\in P(V)$. We represent each $\pi$ by having $i$ lie above $j$ only if $j\le_\pi i$. For example, the square in the upper left corner represents the poset with relations $1\le_\pi 2 \le_\pi 3$ while the bottom most square represents the empty poset.
    Posets $(V,\le_1)$ and $(V,\le_2)$ are connected by an edge whenever there exists a unique pair $i,j \in V$ such that $i \le_1 j$, but $i \incomp_2 j$.
    (b) shows a relabeling of $\cH_{\cP(V)}$ by replacing each $\pi$ with $G_\pi$ when $\cI(\bbP) = \{X_1 \ci_\bbP X_2 \mid X_3\}$. The number of edges of each $G_\pi$ is indicated in the bottom right corner of the square containing it. The direction of edges indicates a strict decrease in the number of edges from one graph to the next.
    A possible path that algorithm~\ref{alg:gspo1} could take starting at the bottom square is highlighted in blue,
    with the graph returned colored green.
	}\label{fig:posetOfPosetsAlgDemonstration}
\end{figure*}

Perhaps the most natural graph structure on $\mathcal{P}(V)$ is known as the \emph{Hasse diagram of the poset of posets} \citep{bouc2013poset},
which we denote by $\mathcal{H}_{\mathcal{P}(V)}$.
One obtains this by adding an edge to connect posets $(V,\le_1)$ and $(V,\le_2)$
whenever there exists a unique pair $i,j \in V$ such that $i \le_1 j$,
but $i \incomp_2 j$.
Figure~\ref{fig:posetofposets} gives an example of $\mathcal{H}_{\mathcal{P}(V)}$
when $V = \{1,2,3\}.$
For more details about Hasse diagrams, see \cite{stanley2011enumerative}.
%and for more on the poset of posets, see \cite{bouc2013poset}.
%For more on Hasse diagrams, which are a basic object of study in combinatorics, see e.g. \cite{stanley2011enumerative}.

% Algorithm~\ref{alg:gspo1} describes a greedy search procedure over $\cH_{\cP(V)}.$ Beginning at an empty poset $\pi:=\pi_0$, the algorithm computes $G_{\pi}$ to obtain a minimal IMAP of $\bbP$. It then does a depth-first search (DFS) in $\cH_{\cP(V)}$ rooted at $\pi$ to find another poset $\tau$ such that $|G_\pi| > |G_\tau|$, where the path of posets $\pi_1:=\pi, \pi_2,\dots, \pi_k = \tau$ in $\cH_{\cP(V)}$ satisfies $|G_{\pi_i}| \ge |G_{\pi_{i+1}}|$, i.e., the sequence of the number of edges of the graphs is weakly decreasing. If such a poset $\tau$ is found, it transitions to $\tau$ by setting $\pi = \tau$ and repeating the DFS at this new $\pi$. When no such $\tau$ is found, it returns $G_\pi$.

% Algorithm~\ref{alg:gspo1} describes a greedy search procedure over $\cH_{\cP(V)}$. Beginning at an any poset $\pi_0$, e.g., the empty poset, the algorithm first computes $G_{\pi_0}$ to obtain a minimal IMAP of $\bbP$.
% Then, it repeatedly searches for a weakly decreasing path to a poset $\tau$ giving a sparser IMAP $G_\tau$, terminating if it cannot do so.

Algorithm~\ref{alg:gspo1} is a greedy search along the edges of $\cH_{\cP(V)}$ to determine a poset $\pi$
yielding the sparsest $G_\pi$.
Figure~\ref{fig:posetofmags} shows an example run of Algorithm~\ref{alg:gspo1} when $\cI(\bbP)=\{X_1\ci_\bbP X_2 \mid X_3\}$, where 
%The figure shows relabeling of the space $\cH_{\cP(V)}$ achieved by replacing 
each poset $\pi$ is replaced by its corresponding $G_\pi$, along with a possible path taken when starting at the empty poset.

\begin{algorithm}[b!]
	\caption{}
	\label{alg:gspo1}
	\begin{algorithmic}
		\State{\textbf{Input:} $\cI(\bbP)$, with $\bbP$ restricted-faithful to $G^*$;   a starting poset $\pi_0$.}
		\State{\textbf{Output:} A minimal IMAP of $G^*$.}
		\State Set $\pi = \pi_0$;
		\State Via depth-first search on $\cH_{\mathcal{P}(V)}$ with root $\pi$,
		find a path $\pi_1:=\pi,\dots,\pi_k:=\tau$ such that $\pi_i$ is adjacent to $\pi_{i+1}$ in $\mathcal{H}_\mathcal{P}$, $|G_{\pi_i}| \ge |G_{\pi_{i+1}}|$
		and $|G_{\pi}| > |G_{\tau}|$. \\
		If such $\pi_k$ exists, set $\pi$ to $\pi_k$, and repeat this step.\\
		Otherwise, return $G_\pi$.
	\end{algorithmic}
\end{algorithm}

As the example in Figure~\ref{fig:posetofmags} shows,
$G_\pi = G_\tau$ can happen for $\pi \neq \tau$.
To achieve better run-time performance, one might optimize directly over the set $\{G_\pi : \pi \in \mathcal{P}(V)\}$ rather than $\cP(V)$,
thus avoiding moving between posets that give rise to the same graph, similar as in GSP~\citep{solus2017consistency,mohammadi_2018}.
We propose to do this by moving from $G_\pi$ to $G_{\po{G'}}$ where $G'$ is obtained from $G_\pi$ via a legitimate mark change,
the definition of which we now restate.
\begin{defn}[\cite{zhang2012transformational}]
    \label{defn:lmc}
    Given a DMAG $G$, a \emph{legitimate mark change of $G$} is the process of turning
    an edge $i \rightarrow j$ to $i \leftrightarrow j$, or vice-versa, when
    \begin{enumerate}
    \vspace{-0.2cm}
        \item there is no directed path from $i$ to $j$ aside from possibly $i\rightarrow j$;
        \vspace{-0.1cm}
        \item if $k \rightarrow i$, then $k \rightarrow j$. If $k \leftrightarrow i$, then $k \leftrightarrow j$ or $k \rightarrow j$;
        \vspace{-0.1cm}
        \item there is no discriminating path $\langle k,\dots,i,j\rangle$.
        \vspace{-0.2cm}
    \end{enumerate}
\end{defn}
\cite{zhang2012transformational} showed that DMAGs $G$ and $H$ are Markov equivalent if and only if
$G$ can be transformed into $H$ via a sequence of legitimate mark changes.
This is analogous to the result by \cite{chickering1995transformational} that DAGs $G$ and $H$ are Markov equivalent
if and only if $G$ can be transformed into $H$ via a sequence of covered edge flips, which are exactly the moves used by GSP~\citep{solus2017consistency,mohammadi_2018}.
Using this notion of edge change gives a different search space,
defined in terms of the IMAPs $G_\pi$.
Namely, given a distribution $\bbP$,
define $\mathcal{L}_\bbP$ to be the directed graph with vertex set
$\{G_\pi: \pi \in \mathcal{P}(V)\}$
with an arc from $G_\pi$ to $G_\tau$ when there exists a graph $G'$,
obtainable from $G_\pi$ via a single legitimate mark change, such that $\tau = \po(G')$.

Figure~\ref{fig:space} shows the outgoing edges of a particular minimal IMAP $G_\pi$ in $\mathcal{L}_\bbP$ when $\bbP$
is faithful to $G^*$ of Figure~\ref{fig:trueGraph}.
As shown, there are two possible legitimate mark changes that can be performed on $G_\pi$ shown as dashed. Changing the bidirected dashed edge, for example, would result in $G_1'$ with $\tau_1 = \po(G_1')$. Hence, there is an outgoing edge from $G_\pi$ to $G_{\tau_1}$ in $\mathcal{L}_\bbP$.

Algorithm \ref{alg:gspo} is the resulting greedy search for the sparsest $G_\pi$ over $\mathcal{L}_\bbP$.
We call this algorithm the \emph{greedy sparsest poset algorithm} (\emph{GSPo}).
We conjecture, supported by simulations on the order of 100,000s of examples (see Appendix~\ref{appendix:sparsity_plot}), that GSPo is consistent under the restricted-faithfulness assumption (using a sufficiently large depth $d$ in the search), i.e., it yields a DMAG that is Markov equivalent to $G^*$ no matter the starting point.
This conjecture generalizes the consistency result proven for GSP in the fully observed setting~\citep{solus2017consistency}.
\begin{figure}[t!]
    \centering
    \includegraphics[width=\linewidth]{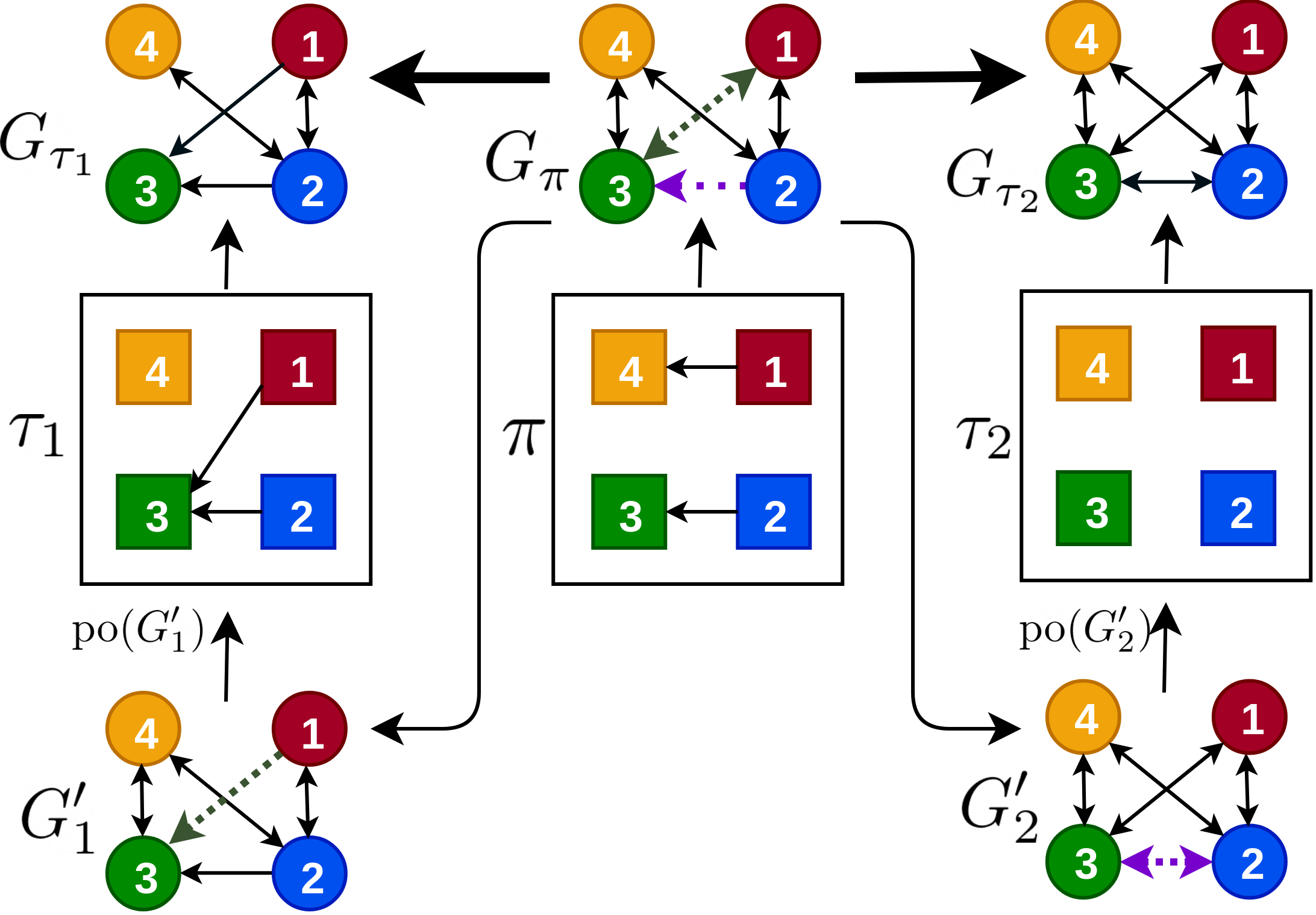}
    \caption{Example of the outgoing edges (in bold) of the node $G_\pi$ in $\mathcal{L}_\bbP$ where
    $\bbP$ is faithful to $G^*$ from Figure~\ref{fig:notAnIMAP}.
    The graphs $G_1'$ and $G_2'$ are obtained from $G_\pi$ via legitimate mark changes of the colored dashed edges.
    The posets $\tau_1,\tau_2$ are $\po(G_1'),\po(G_2')$ respectively
    so $\mathcal{L}_\bbP$ has edges from $G_\pi$ to $G_{\tau_1}$ and $G_{\tau_2}$.}
    \label{fig:space}
    \vspace*{-\baselineskip}
\end{figure}

\begin{conjecture}\label{conj:gspo}
    Let $\bbP$ be a probability distribution that is restricted-faithful to a DMAG $G^*$.
    %\textcolor{red}{faithfulness? This conjecture needs to be written much more carefully. Use the same language as in Section 2.3 for the GSP result.}
    If $\pi_0$ is any poset,
    then there exists a directed path $\pi_0\rightarrow \pi_1\rightarrow \dots\rightarrow \pi_k$ in $\mathcal{L}_\bbP$
    such that $G_{\pi_k}$ is sparsest,
    and such that $\pi_i$ always has weakly fewer edges than $\pi_{i-1}$.
\end{conjecture}
\subsection{Implementation}\label{subsection:implementation}
\vspace{-5pt}A crucial practical consideration for GSPo is the choice of the starting poset $\pi_0$,
since a sparser initial IMAP would be favorable.
%requires fewer CI tests to reach the true Markov equivalence class.
The empty poset $\emptyset$ provides a simple starting place, with $G_\emptyset = \{ i \leftrightarrow j \mid X_i \not\ci_\bbP X_j \}$, but in general will not be sparse. An effective alternative is to start at a sparse DAG that is a minimal IMAP (e.g., given by a permutation), either by running a DAG-learning algorithm such as GSP or by simply using the same starting heuristic as GSP based on the minimum-degree (MD) algorithm \citep{solus2017consistency}. We compare these initialization schemes in Section~\ref{section:simulation}.

\section{EXPERIMENTAL RESULTS}
\label{section:simulation}
In this section, we compare the performance of GSPo to FCI and FCI+ in recovering DMAGs from samples of the observed nodes. In each simulation, we sample 100  Erd{\"o}s-R{\'e}nyi DAGs on $p + K$ nodes with $s$ expected neighbors per node, then form DMAGs by marginalizing over the first $K$ nodes, to obtain DMAGs on $p$ nodes. 
If $i,j$ is an edge $i \rightarrow j$ in the DAG, we assign an edge weight $w_{ij}$ drawn uniformly at random from  $[-1, -.25] \cup [.25, 1]$; we set $w_{ij} = 0$ otherwise. Finally, we generate $n$ samples from the structural equation model $X = W^\top X + \epsilon$ where $\epsilon \sim \cN(0, I_{K+p})$ and remove the first $K$ columns of the data matrix.

\begin{algorithm}[t!]
	\caption{\textsc{Greedy Sparsest Poset (GSPo)}}
	\label{alg:gspo}
	\begin{algorithmic}
		\State{\textbf{Input:} $\cI(\bbP)$, with $\bbP$ restricted-faithful to $G^*$;  starting poset $\pi_0$; maximum depth $d$.}
		\State{\textbf{Output:} A minimal IMAP of $\bbP$.}
		\State Set $\pi = \pi_0$;
		\State Via depth-first search with root $\pi$ and depth at most $d$,
		find path $\pi_0,\dots,\pi_k$ such that $\pi_i$ and $\pi_{i+1}$ 
		are adjacent in $\mathcal{L}_{\bbP}$, 
		$|G_{\pi_i}| \ge |G_{\pi_{i+1}}|$
		and $|G_{\pi_0}| > |G_{\pi_k}|$.\\
		If such $\pi_k$ exists, set $\pi$ to $\pi_k$, and repeat this step.\\
		Otherwise, return $G_\pi$.
	\end{algorithmic}
\end{algorithm}

In each run of GSPo, we set the depth parameter $d$ to $4$, and run the algorithm 5 times for each graph (using different initializations). For DAGs, a depth of 4 has been used to reflect the empirically-observed average size of the MECs \citep{gillispie2001enumerating,solus2017consistency}. Although we are not aware of results on the average size of the MECs of DMAGs, we found little benefit in using values larger than $4$.

\begin{figure*}[h]
	\centering
	\begin{subfigure}[t]{.31\textwidth}
	    \includegraphics[width=\textwidth]{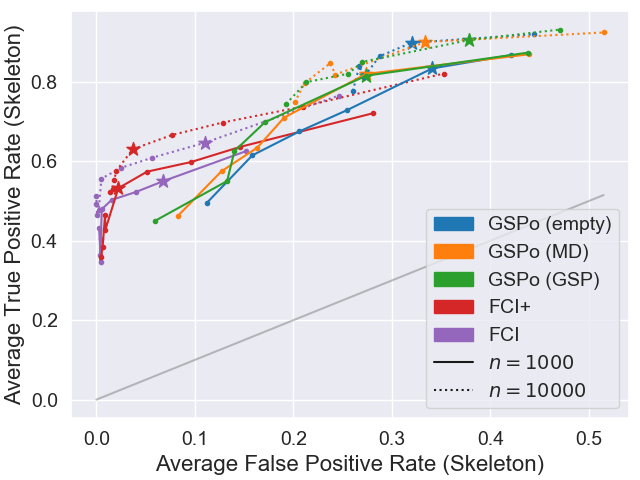}
	    \caption{Skeleton edge recovery (ROC)}
	    \label{fig:roc}
	\end{subfigure}
	~
	\begin{subfigure}[t]{.31\textwidth}
	    \includegraphics[width=\textwidth]{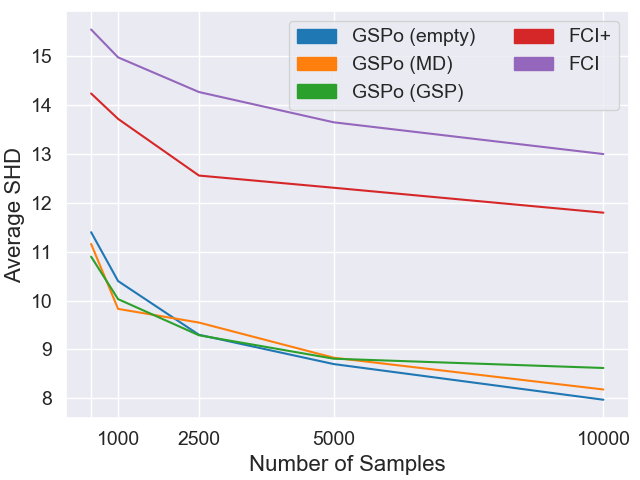}
	    \caption{Skeleton Edge Recovery (SHD)}
	    \label{fig:shd}
	    \end{subfigure}
	\begin{subfigure}[t]{.31\textwidth}
        \includegraphics[width=\textwidth]{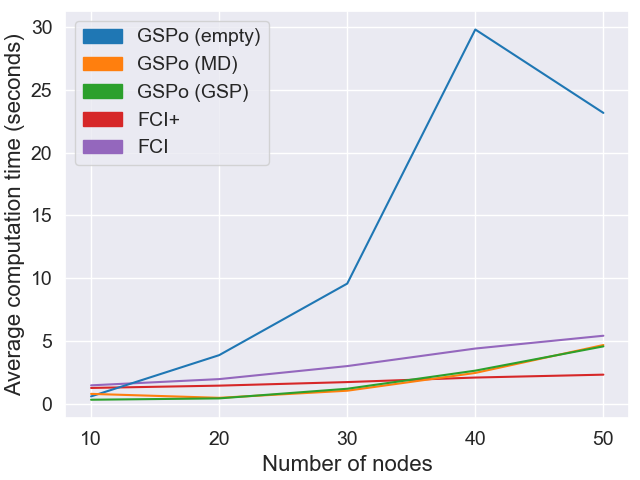}
	    \caption{Median Runtime}
	    \label{fig:computation-time}
	\end{subfigure}
	\caption{In (a) and (b), $p = 10$, $K = 3$, and $s = 3$. In (a), each variant of GSPo was run on 8 $\alpha$ values from $10^{-10}$ to $.7$, and each variant of FCI was run on 7 $\alpha$ values from $10^{-20}$ to $.5$. The best $\alpha$ for each algorithm was selected for (b); the corresponding point is marked by $\star$ in (a). These values were $\alpha=.1$ for each variant of GSPo and for FCI+, and $\alpha=.7$ for FCI. In (c), $p =$ 10, 20, 30, 40, 50, $K = 3$, and $s = 3$. Again the best $\alpha$ was selected for (c), except for FCI, which was run with $\alpha=10^{-3}$ since higher $\alpha$ values were extremely slow.} 
	\vspace{-0.2cm}
    \label{fig:performance}
\end{figure*}

In Figure \ref{fig:performance}, we chose $p = 10$, $K = 3$, and $s = 3$. %Since some nodes are made adjacent by marginalization, 
The resulting graphs have on average about 4 neighbors per node, and have varying proportions of bidirected edges, from 0\% bidirected to 75\% bidirected, with roughly 30\% bidirected on average.

Figure \ref{fig:roc} shows performance of GSPo with three initialization schemes as compared to FCI and FCI+ on recovering the skeleton of the true MAG.
Regardless of the initialization scheme,
GSPo generally estimates denser graphs than FCI and FCI+,
with the densest graphs estimated when starting at the empty poset.
The performance of initializing GSPo by the MD algorithm and GSP are comparable, so for simplicity we recommend initializing by the MD algorithm. While FCI and FCI+ achieve better performance in the low false positive rate regime, GSPo begins to surpass FCI and FCI+ in the middle regime. This indicates that even with a large number of samples, FCI(+) suffers from near-faithfulness violations, which leads to mistakenly removing edges. % based on its repeated series of hypothesis tests.
ROC curves for $p = 50$ nodes are reported in Appendix~\ref{appendix:additional-simultations}, with similar findings.

Figure \ref{fig:shd} shows the structural Hamming distance (SHD)\footnote{the SHD between two undirected graphs is equal to the minimum number of edge additions/deletions required to transform from one graph to another} of the skeleton of the true DMAG to the skeleton of the DMAG estimated by each algorithm. For each algorithm, the value of $\alpha$ was picked from among the values used in Figure \ref{fig:roc} in order to minimize the average SHD over all sample sizes; the corresponding values are marked by stars on the ROC curves. All variants of GSPo outperform both variants of FCI for all sample sizes in terms of SHD.

Figure \ref{fig:computation-time} shows the median computation time required for each algorithm for graphs of varying number of vertices.
Average computation time is in Appendix \ref{appendix:additional-simultations}.
For each algorithm, we chose the parameter $\alpha$ based on the best-performing value in Figure \ref{fig:shd}; FCI we were limited to $\alpha = 10^{-3}$ due to its poor scaling for dense graphs.
Thus, the median runtime for FCI is a conservative lower bound.
We observe that GSPo with GSP initialization is  faster than FCI or FCI+ for small graphs,
but slower than FCI+ as the number of nodes increases.
Given that CI tests in the construction of $G_\pi$ involve all ancestors of pairs of nodes,
this poor scaling is expected.
Fortunately, this suggests that improvements along the lines of those in FCI+ may bring the scaling of GSPo in line with that of FCI+.

\section{DISCUSSION}
\label{section:discussion}
We provided a new characterization of the Markov equivalence class of a DMAG in terms of the set of sparsest minimal IMAPs, which allows structure learning in the presence of latent confounders to be expressed as a discrete optimization problem. To restrict the search space for this problem, we introduced a map from posets to minimal IMAPs whose image contains the true DMAG.
Then, we proposed a greedy algorithm in the space of minimal IMAPs to determine the sparsest minimal IMAP and hence a graph that is Markov equivalent to the true DMAG. This algorithm extends the Greedy Sparsest Permutation algorithm~\citep{solus2017consistency} for learning DAGs to the setting with latent confounders, thereby providing a general hybrid approach for causal structure discovery in this setting. We also demonstrated that it outperforms the current constraint-based methods FCI and FCI+ in some relevant settings.

Consistency of our greedy algorithm remains an open question, and is an interesting issue for future work. Furthermore, it may be possible to improve the statistical and computational performance of GSPo through modifications such as: more efficiently obtaining minimal IMAPs after legitimate mark changes, using dynamic connectivity algorithms to keep track of ancestral relations, and better heuristics for initialization.

By introducing a method for structure learning for DMAGs that is not a variant of FCI, we open the door to comparisons between the behavior of different types of methods on issues besides just statistical and computational performance, such as behavior of the algorithms under misspecification of parametric or modeling assumptions (e.g., non-i.i.d. data or non-Gaussianity when using partial correlation tests). %Given the importance of having a multiplicity of methods for a given task, 
It would also be interesting to use the idea of an ordering-based search as provided in this paper for the problem of learning general MAGs (i.e., including selection bias). To the best of our knowledge, there is no known transformational characterization for Markov equivalence classes of general MAGs yet, which is a key ingredient in the development of such a greedy algorithm.

\newpage
\subsubsection*{Acknowledgements}
Daniel Irving Bernstein was funded by an NSF Mathematical Sciences Postdoctoral Research Fellowship (DMS-1802902). Basil Saeed was partially supported by the Abdul Latif Jameel Clinic for Machine Learning in Health at MIT.
Chandler Squires was supported by an NSF Graduate Research Fellowship and an MIT Presidential Fellowship. Caroline Uhler was partially supported by NSF (DMS-1651995), ONR (N00014-17-1-2147 and N00014-18-1-2765), IBM, a Sloan Fellowship and a Simons Investigator Award.

\bibliography{bib.bib}

\clearpage
\newpage
\section*{APPENDIX}

\appendix 
\newcommand{\snum}{S}
\renewcommand{\theequation}{\snum.\arabic{equation}}
\counterwithin{algorithm}{section}
\counterwithin{figure}{section}

\section{Graph Theory}
\label{appendix:graph-theory}
This section provides additional graph-theoretic notations that are standard in the literature and are provided for ease of access.
Let $G = (V,D,B)$ be a graph.
%A directed edge will be denoted $i \rightarrow j$,
%a bidirected will be denoted $i \leftrightarrow j$, where $i, j \in V$ are distinct and at most one type of edge exists between any pair $i, j$.
If there is any edge between $i$ and $j$, they are called \emph{adjacent} which we may denote $i \sim j$.
Otherwise they are called \emph{non-adjacent} and we write $i \notadj j$.
We will use $\circ$ as a ``wildcard" for edge marks, i.e. $i \rightarrowcirc j$ denotes that either $i \rightarrow j$ or $i \leftrightarrow j$.
We will use subscripts on these vertex relations as a shorthand way to indicate
the presence or absence of an edge, or the presence of a particular kind of edge.
For example, $i \leftrightarrow_G j$ and $k \notadj_G l$ respectively indicate that $G$ has a bidirected edge between $i$ and $j$,
and no edge between $k$ and $l$.
A graph with only directed edges is called a \emph{directed graph}.

A \emph{path} $\gamma = \langle v_1, v_2, \ldots, v_k \rangle$ is a sequence of distinct nodes that such that $v_i$ and $v_{i+1}$ are adjacent.
A \emph{cycle} is a path together with any type of edge between $v_k$ and $v_{k+1} = v_1$.
A path or a cycle is called \emph{directed} if all edges are directed toward later nodes, i.e. $v_i \rightarrow v_{i+1}$.

% Given a graph $G$,
% we define the \emph{parents}, \emph{children}, \emph{spouses}, and \emph{ancestors} of a node $k$ as
% $$
% \pa_G(k) = \{ i \mid i \rightarrow k \}
% $$
% $$
% \ch_G(k) = \{ i \mid k \rightarrow i \}
% $$
% $$
% \spo_G(k) = \{ i \mid i \leftrightarrow k \}
% $$
% % $$
% % \adj_G(k) = \{ i \mid i \sim k \}
% % $$
% $$
% \an_G(k) = \{ i \mid ~\textrm{there is a directed path from}~i~\textrm{to}~k\}
% $$
% % $$
% % \de_G(k) = \{ i \mid ~\textrm{there is a directed path from}~k~\textrm{to}~i\}
% % $$
We extend the notation $\pa_G(i), \spo_G(i),$ and $\an_G(i)$
to allow arguments that are subsets of vertices by taking unions. For example, when $S \subseteq V$, we have
$$
\pa_G(S) := \cup_{i \in S} \pa_G(i).
$$
We add an asterisk to denote that the arguments are not included in the set, e.g.
$$
\pa^*_G(S) := \pa_G(S) \setminus S.
$$

%A graph with no directed cycles is called \emph{acyclic}.
%An \emph{ancestral} graph is acyclic and satisfies $\spo_G(k) \cap \an_G(k) = \emptyset$ for all $k \in V$.

The \emph{colliders} on a path $\gamma$ are the nodes where two arrowheads meet, i.e., $v_i$ is a collider if $v_{i-1} \rightarrowcirc v_i \leftarrowcirc v_{i+1}$.
A triple of nodes $(i, j, k)$ is called a \emph{v-structure} if $j$ is a collider on the path $\langle i, j, k \rangle$ and $i \notadj k$.

\section{Proof of Proposition \ref{prop:construction}}
\label{proof_prop1}
We will prove Proposition \ref{prop:construction} via a sequence of intermediate Lemmas.
Since our goal is to prove that all the $m$-separation statements of a given DMAG
are satisfied by a given $\bbP$,
it will be helpful to have the following lemma which reduces the number of $m$-separation statements we must consider.

\begin{lemma}\label{lemma:pairwiseMarkov}
    Let $G^*$ and $H$ be DMAGs.
    Then $G^* \le H$ if and only if whenever $i\notadj_H j$,
    $i$ is $\an_{H}(\{i,j\}$-separated from $j$ in $G^*$, i.e.
    $i \ci_{G^*} j \mid \an_{H}(\{i,j\})$.
\end{lemma}
\begin{proof}
    This is an immediate consequence of Theorem~3 in (Sadeghi and Lauritzen, 2014).
    % \citep{sadeghi2014markov}
\end{proof}

Throughout the rest of this section,
let it be understood that $G^*$ is a DMAG that is restricted-faithful to some fixed joint distribution $\bbP$.
We will not repeat this assumption.
Moreover, we will suppress $\bbP$ in our notation and write
$G_\pi$ instead of $G_\pi$ and $AG(\pi)$ instead of $AG(\pi,\bbP)$.
Also, note that when $H$ is a DMAG, $\po(H) = \po(\overline{H})$
since $\overline{H}$ is obtained from $H$ by adding only bidirected edges \citep{richardson2002ancestral}.
We will make repeated tacit use of this fact.

\begin{lemma}\label{lemma:imapIfFixedPoint}
    Let $\pi$ be a partial order on the random variables of $\bbP$
    such that $G_\pi = \overline{AG(\pi)}$.
    Then $G_\pi$ is an IMAP of $\bbP$.
\end{lemma}
\begin{proof}
    Lemma \ref{lemma:pairwiseMarkov} implies that it suffices to show that
    whenever $i\notadj_{G_\pi} j$, $X_i \ci_\bbP X_j \mid X_{\an_{G_\pi}^*(i,j)}$.
    So assume $i \notadj_{G_\pi} j$.
    Since $G_\pi = \overline{AG(\po(AG(\pi)))}$,
    $i \notadj_{G_\pi} j$ implies $X_i \ci_{\bbP} X_j \mid X_{\pre^*_{\po(AG(\pi))}(i,j)}$.
    But now we are done since $\pre^*_{\po(AG(\pi))}(i,j)=\an_{AG(\pi)}^*(i,j)=\an_{\overline{AG(\pi)}}^*(i,j)$
    and we are assuming $G_\pi = \overline{AG(\pi)}$.
\end{proof}

\begin{lemma}\label{lemma:fixedPoint}
    Let $\pi$ be a partial order on the random variables of $\bbP$.
    Then $\po(G_\pi) = \po(AG(\pi))$.
\end{lemma}
\begin{proof}
    We must show 
    $$\po(AG(\po(AG(\pi)))) = \po(AG(\pi))$$
    If $i \le j$ in $\po(AG(\po(AG(\pi))))$,
    then there exists a directed path $i = i_0\rightarrow \dots \rightarrow i_k = j$
    in $AG(\po(AG(\pi)))$
    and so $i = i_0 \le \dots \le i_k = j$ in $\po(AG(\pi))$.
    
    We now proceed to show that if $i \le j$ in $\po(AG(\pi))$,
    then the same is true in $\po(AG(\po(AG(\pi))))$.
    We do this by showing that if $i \rightarrow_{AG(\pi)} j$, then $i \rightarrow_{AG(\po(AG(\pi)))} j$.
    So for the sake of contradiction, assume that $i \rightarrow_{AG(\pi)} j$ but not $i \rightarrow_{AG(\po(AG(\pi)))} j$.
    By the definition of $AG$,
    this implies that $i \notadj_{AG(\po(AG(\pi)))} j$ and so $i$ is m-separated from $j$ given $\an_{AG(\pi)}^*(i,j)$ in $G^*$.
    But $i\rightarrow_{AG(\pi)} j$ implies that $i$ is m-connected to $j$ given $\pre_\pi^*(i,j)$ in $G^*$.
    Let $P$ be an m-connecting path from $i$ to $j$ given $\pre_{\pi}^*(i,j)$ in $G^*$.
    Since $\an_{AG(\pi)}^*(i,j) \subseteq \pre_{\pi}^*(i,j)$, we can write
    \[
        \pre_{\pi}^*(i,j) = \an_{AG(\pi)}^*(i,j) \cup S
    \]
    for some nonempty set $S$, disjoint from $\an_{AG(\pi)}^*(i,j)$.
    Since $i$ is m-separated from $j$ given $\an_{AG(\pi)}^*(i,j)$ in $G^*$,
    $P$ must contain a collider with a descendent in $S$, but no descendant in $\an_{AG(\pi)}^*(i,j)$.
    Let $d$ be such a collider that is closest to $j$ along $P$
    and let $s$ be a $\po(G^*)$-minimal descendent of $d$ from $S$.
    
    We now construct a path $Q$ in $G^*$ that m-connects $j$ and $s$ given $\pre_\pi^*(j,s)$.
    Since $S \subseteq \pre_\pi(j)$, this would imply existence of the edge $s\rightarrow_{AG(\pi)} j$,
    contradicting $s \in S$.
    If $s = d$, we let $Q$ be the subpath of $P$ from $j$ to $s$.
    Otherwise, we let $Q$ be obtained by concatenating the subpath of $P$ from $j$ to $d$,
    followed by a directed path from $d$ to $s$.
    Since $P$ is m-connecting given $\pre_{\pi}^*(i,j)$ and $i,s \le j$ in $\pi$,
    it follows that when $Q$ is a subpath of $P$,
    $Q$ is m-connecting given $\pre_{\pi}^*(j,s)$.
    When $Q$ additionally has a directed path from $d$ to $s$,
    $Q$ is m-connecting given $\pre_{\pi}^*(j,s)$ since the non-$P$ segment has no colliders,
    and assumptions on $d$ and $\po(G^*)$-minimality of $s$ imply that no element of this segment is in $\an_{AG(\pi)}^*(s,j)$.
\end{proof}

\begin{proof}[Proof of Proposition \ref{prop:construction}]
    Define $\tau := \po(AG(\pi))$.
    Since $\po(H) = \po(\overline H)$ for any DMAG $H$,
    we have
    \begin{align*}
        G_\tau &= \overline{AG(\po(G_\pi))}.
    \end{align*}
    Lemma \ref{lemma:fixedPoint} implies that this is equal to $\overline{AG(\po(AG(\pi)))}$,
    which is equal to both $G_\pi$ and $\overline{AG(\tau)}$.
    Thus we have shown that $G_\pi = G_\tau = \overline{AG(\tau)}$
    and so Lemma \ref{lemma:imapIfFixedPoint} implies that $G_\pi$ is an IMAP of $\bbP$.
    
    We now show that $G_\pi$ is a \emph{minimal} IMAP of $\bbP$,
    i.e. that removing any edge results in a directed ancestral graph that is either not maximal,
    or not an IMAP of $\bbP$.
    Let $i,j$ be such that $i\sim_{G_\pi} j$
    and let $G'$ be the graph obtained from $G_\pi$ by removing the edge between $i$ and $j$.
    If $G'$ is still maximal,
    then Lemma \ref{lemma:pairwiseMarkov} implies that
    $i$ is m-separated from $j$ given $\an_{G'}^*(i,j)$ in $G'$.
    If $G^* \le G'$, then $i$ is m-separated from $j$ given $\an_{G'}^*(i,j)$ in $G^*$.
    Note that $\an_{G'}^*(i,j) = \an_{G_\pi}^*(i,j)$,
    and that Lemma \ref{lemma:imapIfFixedPoint} implies that $\an_{G_\pi}^*(i,j) = \pre_{\po(AG(\pi))}^*(i,j)$.
    But if $i$ were $\pre_{\po(AG(\pi))}^*(i,j)$-separated from $j$ in $G^*$,
    then $X_i \ci_\bbP X_j \mid X_{\pre_{\po(AG(\pi))}^*(i,j)}$ and so $i\notadj_{AG(\pi)} j$.
    This would imply that $AG(\pi)$ is a subgraph of $G'$.
    Since $G'$ is maximal, $G_\pi$ would be a subgraph as well
    contradicting $i\sim_{G_\pi} j$.
\end{proof}

\section{Proof of Theorem~\ref{thm:sparsest_imap}}
\label{proof_thm1}

We begin by proving the following lemma, which extends classic results for the case of DAGs and deals with discriminating paths.

\begin{lemma}
\label{lemma:faithful-imap-relations}
Let $G^*$ and $H$ be DMAGs and let $\bbP$ be a distribution that is Markov to both $G^*$ and $H$.%, i.e., $\cI(H),\cI(G^*) \subseteq \cI(\bbP)$.
If $\bbP$ is adjacency-faithful to $G^*$, then
\begin{itemize}
    \item[(a)] $\skel(G^*) \subseteq \skel(H)$.
\end{itemize}
If $\bbP$ is furthermore orientation-faithful to $G^*$, then
\begin{itemize}
    \item[(b)] If $i \rightarrowcirc k \leftarrowcirc j$ is a v-structure in $G^*$, then either $i \rightarrowcirc k \leftarrowcirc j$ is a v-structure in $H$ or $i \sim_H j$.
    \item[(c)] If $i \rightarrowcirc k \leftarrowcirc j$ is a v-structure in $H$, then either $i \rightarrowcirc k \leftarrowcirc j$ is a v-structure in $G^*$, or $i \notadj_{G^*} k$ or $j \notadj_{G^*} k$.
\end{itemize}
Finally, if $\bbP$ is also discriminating-faithful to $G^*$, then
\begin{itemize}
    \item[(d)] If $\gamma:=\langle i,\dots,k,j\rangle$ is a discriminating path in both $H$ and $G^*$,
    then $k$ is a collider in $\gamma$ in $H$ iff $k$ is a collider in $\gamma$ in $G^*$.
\end{itemize}
\end{lemma}

\label{appendix:proof-lemma-1}
\begin{proof}
(a) If $i \notadj_H j$, then by the pairwise Markov property \citep{richardson2002ancestral}, $X_i \ci_\bbP X_j \mid X_{\an^*_H({i, j})}$, and by adjacency-faithfulness, $i \notadj_{G^*} j$ in $G^*$.

(b) Let $i \notadj_H j$, so $X_i \ci_\bbP X_j \mid X_{\an^*_H(\{i, j\})}$. Suppose $k$ is a parent of either $i$ or $j$. Since $k \in \an^*_H(\{i, j\})$, $i$ is m-connected to $j$ in $G^*$ given $\an^*_H(\{i, j\})$ by the path $i \rightarrowcirc k \leftarrowcirc j$, and thus $X_i \not\ci_\bbP X_j \mid X_{\an_H(\{i, j\})}$ by orientation faithfulness. Hence, $H$ is not an I-MAP of $\bbP$.

(c) Suppose $i \sim_G k$ and $j \sim_G k$. We have $X_i \ci_\bbP X_j \mid X_{\an_H^*(\{i, j\})}$, and thus by orientation faithfulness $i$ and $j$ are m-separated given $\an^*_H(\{i, j\})$ in $G^*$. Since $H$ is ancestral, $k \not\in \an^*_H(\{i, j\})$. Thus, to ensure the required m-separation in $G$, $k$ must be a collider in $G$ on the path $i - k - j$.

(d) Assume $\gamma = \langle i,C_1,\dots,C_l,k,j\rangle$.
If $k$ is a non-collider in $\gamma$ in $G^*$, then $i$ is m-connected to $j$ given $S$ for every $S$ containing $C_1,\dots,C_l$ but not $k$.
Discriminating faithfulness implies $X_i \not\ci_\bbP X_j \mid X_{S}$ for every such $S$.
Then $k$ must also be a non-collider in $\gamma$ in $H$,
since otherwise there would exist some $S$ containing $C_1,\dots,C_l$ but not $k$ such that $i$ is m-separated from $j$ given $S$ in $H^*$,
contradicting $\mathcal{I}(H) \subseteq \mathcal{I}(\bbP)$.
If $K$ is a collider in $\gamma$ in $G^*$,
then $i$ is m-connected to $j$ given $S$ for every $S$ containing $C_1,\dots,C_l,k$.
Again, discriminating faithfulness implies $X_i \not\ci_\bbP X_j \mid X_S$ for every such $S$.
Then $K$ must also be a collider in $\gamma$ in $H$,
since otherwise there would exist some $S$ containing $C_1,\dots,C_l,k$ such that $i$ is m-separated from $j$ given $S$ in $H^*$.
\end{proof}

We proceed to proving the theorem.
\begin{proof}[Proof of Theorem \ref{thm:sparsest_imap}]
(a) is implied by Lemma \ref{lemma:faithful-imap-relations}(a).

Since restricted faithfulness implies adjacency faithfulness, $\skel(G) = \skel(G^*)$.
It remains to show that $G$ and $G^*$ have the same v-structures, and that if $\gamma$ is a discriminating path for $k$ in both $G$ and $G^*$, then $k$ is a collider on $\gamma$ in $G$ iff it is a collider on $\gamma$ in $G^*$.

Equality of skeletons together with Lemma \ref{lemma:faithful-imap-relations}(b) and (c) imply that $G$ and $H$ have the same v-structures.
If $\gamma:=\langle i, C_1,\dots,C_l,k,j\rangle$ is a discriminating path in both $G^*$ and $G$,
then Lemma \ref{lemma:faithful-imap-relations}(d) implies that $k$ is a collider in $\gamma$ in $G^*$ iff $k$ is a collider in $\gamma$ in $G$.
% If $k$ is a collider in $G$ but not $G^*$,
% then there exists $S \subseteq [p]\setminus\{i,j,k\}$ containing $C_1,\dots,C_l$
% such that $i$ is $S$-separated from $j$ in $G$ but not $G^*$.
% Discriminating faithfulness of $G^*$ implies that $i \not\ci_\bbP j | S$ and so $G$ is not an IMAP of $\bbP$.
% If $k$ is a collider in $G^*$ but not $G$,
% then the an analogous argument applies with some $S\subseteq [p]\setminus\{i,j\}$ containing $k,C_1,\dots,C_l$.
\end{proof}

\section{Proof of Proposition 2}
\label{proof_prop2}
\begin{proof}
It is sufficient to show this for $G = G^*$, since Markov equivalence implies that $\cI(G) = \cI(G^*).$ 
Suppose $G = (V,D,B)$. 
Let $\pi = \po(G)$. 
We have already shown that $G_\pi$ is an IMAP. Therefore, it is sufficient to show the converse, i.e., that 
if $X_i \ci_\bbP X_j \mid S$ then $i\ci_{G_\pi} j \mid S$.

By Theorem 4.2 of~\cite{richardson2002ancestral}, for any $i,j\in V$ adjacent, $i\not\ci_{G_\pi} j | \an^*_{G_\pi}(i,j).$
The faithfulness condition would then imply that $X_i\not\ci_{\bbP} X_j | X_{\pre^*_\pi(i,j)}.$
\end{proof}

\section{Conjecture Simulations}\label{appendix:sparsity_plot}
In figure~\ref{fig:sparsity}, we display a scatter plot of the number of edges of the graphs that we tested our algorithm on, without failure. The plot includes over 200,000 points, corresponding to 200,000 generated graphs of various parameters. For each of these, graphs, we have tested the oracle version of our algorithm, i.e., $\cI(\bbP) = \cI(G^*)$, and it converged to a graph in the Markov equivalence class of the true graph. We have not found a single counterexample to the conjecture thus far.

%
%\begin{figure*}[h]
%	\centering
%	\begin{subfigure}[t]{.31\textwidth}
%	    \includegraphics[width=\textwidth]{precision_recall.png}
%	    \caption{}
%	    \label{fig:precision-recall}
%	\end{subfigure}
%	~
%	\begin{subfigure}[t]{.31\textwidth}
%        \includegraphics[width=\texwidth]{roc_large.png}
%	    \caption{}
%	    \label{fig:performance-large}
%	    \end{subfigure}
%	~
%	\begin{subfigure}[t]{.31\textwidth}
%        \includegraphics[width=\textwidth]{time_mean.png}
%	    \caption{}
%	    \label{fig:time-mean}
%	\end{subfigure}
%	\caption{\ref{fig:precision-recall}: Average performance over 100 MAGs for each algorithm, when $p = 50$, $K = 12$, and $s = 3$. Each variant of GSPo was run on 8 $\alpha$ values from $10^{-10}$ to $.7$, and each variant of FCI was run on 7 $\alpha$ values from $10^{-20}$ to $.5$.  
%	\ref{fig:performance-large}: Average performance over 100 MAGs for each algorithm, when $p = 50$, $K = 12$, and $s = 3$. Each variant of GSPo was run on 8 $\alpha$ values from $10^{-10}$ to $.7$, and each variant of FCI was run on 7 $\alpha$ values from $10^{-20}$ to $.5$.
%	\ref{fig:time-mean}
%	Average runtime over 100 MAGs for $p =$ 10, 20, 30, 40, 50, $K = 3$, and $s =3$. Each variant of GSPo and FCI+ were run with $\alpha=.1$, while FCI was run with $\alpha=10^{-3}$ due to the extremely long runtime of higher $\alpha$ values.
%	} 
%	\vspace{-0.2cm}
%    \label{fig:performance}
%\end{figure*}

\begin{figure}[t!]
	\centering
        \includegraphics[width=.98\linewidth]{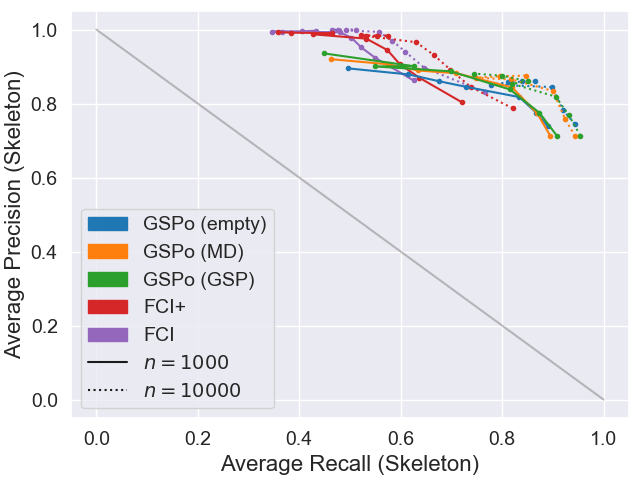}
        \caption{Average performance over 100 MAGs for each algorithm, when $p = 50$, $K = 12$, and $s = 3$. Each variant of GSPo was run on 8 $\alpha$ values from $10^{-10}$ to $.7$, and each variant of FCI was run on 7 $\alpha$ values from $10^{-20}$ to $.5$} 
    \label{fig:precision-recall}
\end{figure}

\begin{figure}[t!]
	\centering
        \includegraphics[width=.98\linewidth]{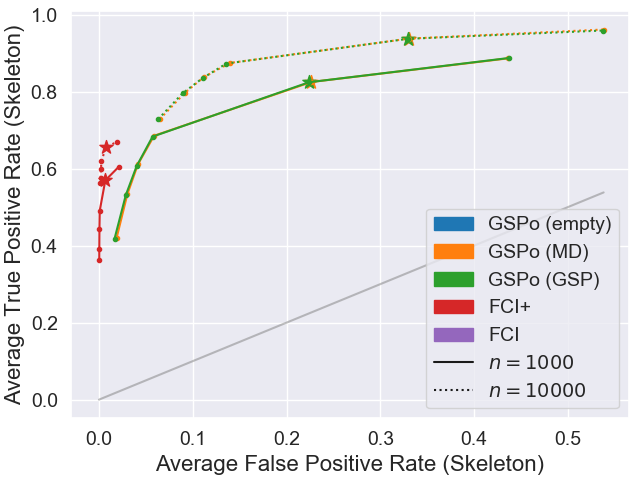}
        \caption{Average performance over 100 MAGs for each algorithm, when $p = 50$, $K = 12$, and $s = 3$. Each variant of GSPo was run on 8 $\alpha$ values from $10^{-10}$ to $.7$, and each variant of FCI was run on 7 $\alpha$ values from $10^{-20}$ to $.5$} 
    \label{fig:performance-large}
\end{figure}

\begin{figure}[t!]
	\centering
        \includegraphics[width=.98\linewidth]{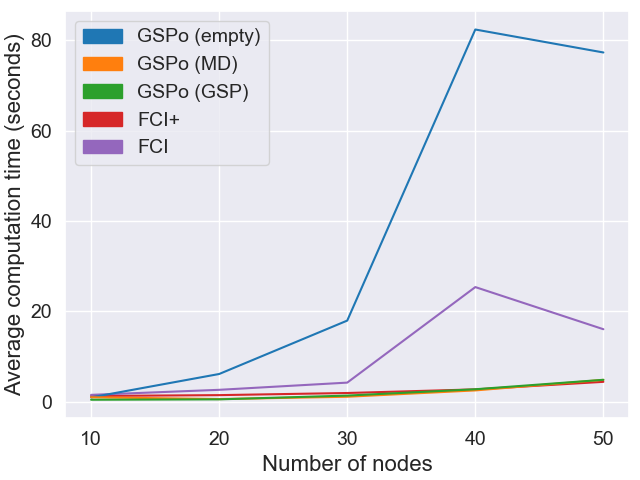}
        \caption{Average runtime over 100 MAGs for $p =$ 10, 20, 30, 40, 50, $K = 3$, and $s =3$. Each variant of GSPo and FCI+ were run with $\alpha=.1$, while FCI was run with $\alpha=10^{-3}$ due to the extremely long runtime of higher $\alpha$ values.} 
    \label{fig:time-mean}
\end{figure}

\begin{figure*}[t]
	\centering
        \includegraphics[width=.98\linewidth]{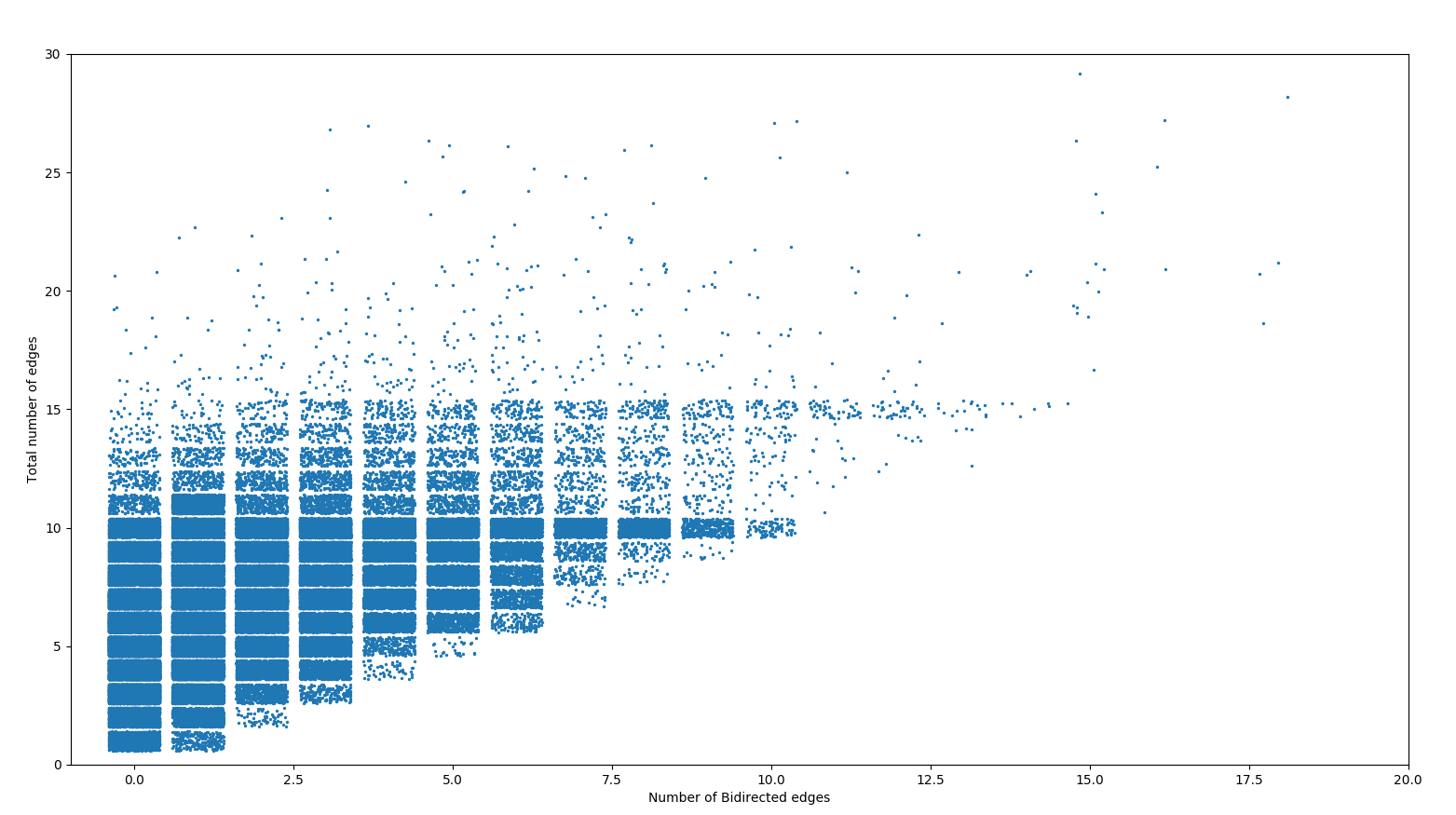}
        \caption{A scatter plot of the number of edges of the graphs that we tested the oracle version of our algorithm on. The plot includes over 200,000 points, representing graphs with varying number of bidirected edges and total number of edges.} 
    % \end{minipage}
    \label{fig:sparsity}
\end{figure*}

\section{Additional Simulations}\label{appendix:additional-simultations}

In this section, we followed the same procedure for DMAG sampling procedure as described in Section \ref{section:simulation}. Fig. \ref{fig:precision-recall} gives the precision-recall curve for the same settings as in Fig. \ref{fig:roc} in Section \ref{section:simulation}.

In Figure \ref{fig:performance-large}, we use $p = 50$ nodes, $K = 12$ latent variables, and $s = 3$ expected neighbors per node in the DAG before marginalization. For 100 graphs, we find that this results in MAGs with an average of 43\% bidirected edges, ranging from 14\% to 71\% bidirected edges, and an average of 5 neighbors per node in the MAGs. Due to the slow runtime of FCI, GSPo with empty initialization, and FCI+ with high $\alpha$ values, our comparison between the algorithms for larger graphs is limited, and mainly serves to demonstrate that GSPo has similar performance on larger graphs for the same range of $\alpha$ values.

In Figure \ref{fig:time-mean}, we use the same set of DMAGs as used in \ref{fig:computation-time}, in particular, $p =$ 10, 20, 30, 40, 50, $K = 3$, and $s = 3$, but report the average computation time instead of the median computation time. We can observe that GSPo with the empty initialization and FCI both have much higher average computation times than median computation times, indicating that they are more susceptible to outlier instances from our sampled MAGs.

\end{document}